\documentclass[10pt,oneside]{amsart}

\usepackage{amsmath,graphicx}
\usepackage{amssymb}
\usepackage{amsfonts, amsthm}
\usepackage{todonotes}
\usepackage{enumitem}
\usepackage{array}
\usepackage{float}
\usepackage{dsfont}

\usepackage[parfill]{parskip}

\usepackage{color}
\usepackage{xcolor}
\usepackage{geometry}

\usepackage{enumitem}

\usepackage[colorlinks=true, pdfstartview=FitV,linkcolor=blue,citecolor=blue, urlcolor=black]{hyperref}
\usepackage{cleveref}
\crefformat{equation}{(#2#1#3)}

\makeatletter
\def\subsection{\@startsection{subsection}{2}%
  \z@{.5\linespacing\@plus.7\linespacing}
{.5\baselineskip}%
  {\normalfont\centering\scshape}%
}
\makeatother

\numberwithin{equation}{section}

\makeatletter
\def\l@subsection{\@tocline{2}{0pt}{2.5pc}{5pc}{}}
\def\l@section{\@tocline{1}{0pt}{0pc}{5pc}{\bfseries}}
\makeatother

\newtheorem{definition}{Definition}
\numberwithin{definition}{section}
\newtheorem{theorem}{Theorem}
\numberwithin{theorem}{section}

\newtheorem*{proposition*}{Proposition}

\newtheorem{lemma}{Lemma}
\numberwithin{lemma}{section}

\newtheorem*{remark*}{Remark}
\newtheorem*{claim*}{Claim}

\crefname{theorem}{Theorem}{Theorems}
\Crefname{theorem}{Theorem}{Theorems}

\crefname{lemma}{Lemma}{Lemmas}
\Crefname{lemma}{Lemma}{Lemmas}

\crefname{proposition}{Proposition}{Propositions}
\Crefname{proposition}{Proposition}{Propositions}

\crefname{corollary}{Corollary}{Corollaries}
\Crefname{corollary}{Corollary}{Corollaries}

\crefname{definition}{Definition}{Definitions}
\Crefname{definition}{Definition}{Definitions}

\crefname{remark}{Remark}{Remarks}
\Crefname{remark}{Remark}{Remarks}

\crefname{proposition*}{Proposition}{Propositions}
\Crefname{proposition*}{Proposition}{Propositions}
\crefname{remark*}{Remark}{Remarks}
\Crefname{remark*}{Remark}{Remarks}
\crefname{claim*}{Claim}{Claims}
\Crefname{claim*}{Claim}{Claims}

\begin{document}

\title[Uniqueness for the fuzzy Landau and multiespecies Landau equations]{
  Uniqueness of bounded solutions to the fuzzy Landau
  and multiespecies Landau equations
  }

\author{F.-U. Caja-Lopez}

\address{Department of Mathematics, The University of Texas at Austin (USA)}
\email{funai.caja@utexas.edu}


\begin{abstract}
  We prove conditional uniqueness of weak solutions to the fuzzy Landau equation and the multiespecies Landau system under suitable integrability assumptions. The results are based on explicit stability estimates in the $2$-Wasserstein distance for a broader class of nonlinear equations with singular coefficients. Interestingly, this class includes the $2$D incompressible Euler equations, the Vlasov--Poisson system, and the Patlak--Keller--Segel model, thereby recovering known uniqueness results within a unified framework. Our approach builds on the stochastic coupling method introduced by Fournier and Guerin for the homogeneous Landau equation in \cite{fournier2010uniqueness_Coulomb}, which we recast in a more analytic form. In addition, we present an alternative argument based on the symmetrization technique of Guillen and Silvestre \cite{GS24}, which was further developed in \cite{paper1}, yielding comparable stability estimates.
\end{abstract}

\keywords{Optimal transport, Wasserstein distance, Landau equation}

\maketitle


\section{Introduction}
This work is motivated by the study of the Landau equation, which models collisional electron dynamics in plasmas:
\begin{equation}\label{eq:inhomo_Landau}
\partial_{t}f+v\cdot\nabla_{x}f = \textnormal{div}_v\int_{\mathbb R^3}\Phi(v-v_*)
\Big(f(t,x,v_*)\nabla_v f(t,x,v)-f(t,x,v)\nabla_{v_*} f(t,x,v_*)\Big)\,dv_*,
\end{equation}
where $\Phi(z)=|z|^{2+\gamma}\mathbb P(z)$ and $\mathbb P(z)=\mathrm{Id}-\frac{z\otimes z}{|z|^2}$ denotes the orthogonal projection onto $\langle z\rangle^\perp$. The parameter $\gamma$ is typically considered in the range $-3\leq \gamma \leq 1$, with the most physically relevant case being $\gamma=-3$, corresponding to Coulomb interactions. The analysis of \eqref{eq:inhomo_Landau} is challenging due to its nonlinear, nonlocal, hypoelliptic structure and the presence of singular coefficients.
We study uniqueness and stability of two variants of \eqref{eq:inhomo_Landau}. First, we consider the \textit{fuzzy Landau equation}
\begin{equation}\label{eq:fuzzy_Landau}
  \partial_{t}f+v\cdot\nabla_{x}f = \textnormal{div}_v\int_{\mathbb R^6} \kappa(x-x_*)\Phi(v-v_*) (\nabla_v-\nabla_{v_*})\big(f(t,x,v) f(t,x_*,v_*)\big)\,dx_*\,dv_*,
\end{equation}
where the kernel $\kappa$ encodes the range of spatial interactions. Similarly to \eqref{eq:inhomo_Landau}, the right-hand side models collisions between particles with position–velocity pairs $(x,v)$ and $(x_*,v_*)$. The key feature distinguishing \eqref{eq:fuzzy_Landau} from the classical Landau equation is that collisions are \emph{delocalized}: particles at different spatial locations may interact whenever $\kappa(x-x_*)>0$. Mathematically, \eqref{eq:fuzzy_Landau} can be viewed as an intermediate model between \eqref{eq:inhomo_Landau} and the spatially homogeneous equation. Note that setting $\kappa=\delta_0$ formally recovers \eqref{eq:inhomo_Landau}, suggesting that the study of \eqref{eq:fuzzy_Landau} may provide insight into the original equation.

The second variant we study is the \emph{homogeneous multiespecies Landau equation}, which describes the evolution of $N$ interacting particle species with the spatial dependence removed:
\begin{equation}\label{eq:multiespecies_Landau}
  \partial_t f^{(i)} = \sum_{j=1}^N q_{ij}(f^{(i)},f^{(j)}),
\end{equation}
where the collision operator between species $i$ and $j$ is given by
\begin{equation*}
  q_{ij}(f,g) := \frac{c_{ij}}{m_{i}}\textnormal{div}_{v^{i}}\int_{\mathbb R^3} \Phi(v^{i}-v^{j}_*)\left(\frac{\nabla_{v^{i}}}{m_{i}}-\frac{\nabla_{v^{j}_*}}{m_{j}}\right)\big(f(t,v^{i})g(t,v^{j}_*)\big)\,dv^{j}_*,
\end{equation*}
Here, $m_i>0$ denotes the mass of species $i$, and $c_{ij}$ is a physical constant given by
\begin{equation*}
  c_{ij} = \frac{|\log \Lambda_c|q_i^2q_j^2n_in_j}{8\pi\varepsilon_0^2},
\end{equation*}
where $|\log \Lambda_c|$ is the Coulomb logarithm, and $\varepsilon_0$ is the permittivity of vacuum. Let us do a brief literature review on \cref{eq:inhomo_Landau,eq:fuzzy_Landau,eq:multiespecies_Landau} as further motivation.

\textbf{The Landau equation.} Several works establish local-in-time existence of solutions to the Landau equation for smooth or bounded initial data; see, for instance, \cite{HeYang2014WellPosedBoltzLandau,henderson2019localExistenceLandau,HendSnelTarfu2020LocalSolLandau}. However, global-in-time existence of classical solutions for general initial data remains a major open problem. Moreover, in connection with hydrodynamic limits and the recent works \cite{MerleRaphRodSze2022ImplosionEuler1,MerleRaphRodSze2022ImplosionEuler2}, which exhibit implosion singularities for the compressible Euler equations, it is expectable that the Landau equation may also develop singular behavior. This phenomenon has recently been observed in \cite{gualdani2026Singularities} for $\gamma > \sqrt{3}$.

Global existence of solutions is currently known only in weaker settings. In particular, renormalized solutions exist globally in time \cite{Lions1994RenormSolLandauBoltzman}, while global smooth solutions have been established only under restrictive assumptions, such as initial data close to equilibrium \cite{Guo2002LandauPeriodicBox,CarraKeblerTristWu2016ExpStabLandau,CarraMisch2017LandauNearMaxw,DuanLiu2021GlobalMildLandau,KimGuo2020L2LinftyLandau,GuoHwangJinOuy2020LandauSpecular,chenNguYang2025LandauBoltzCriticalSpaces} or near vacuum \cite{Luk2019LandauVacuum,chaturvedi2023LandauVacuum}. Another line of work establishes regularity and continuation criteria under uniform control of macroscopic quantities such as the local mass, energy, and entropy
\begin{equation*}
  M(t,x):= \int_{\mathbb R^3}f\,dv,\quad 
  E(t,x):= \int_{\mathbb R^3}f|v|^2\,dv,\quad
  H(t,x):= \int_{\mathbb R^3}f\log f\,dv.
\end{equation*}
The underlying idea in these works is that loss of regularity should be accompanied by a deterioration of these macroscopic quantities; see \cite{MouhotVasseur2019HarnackLandau,henderson2020Cinfty_smoothing,snelson2023continuationLandau,silvestreSnelson2018aPrioriLandau} for further details. We also mention \cite{golding2025hydrodynamicimplosionslandaucoulombequation}, which shows that solutions may be continued as long as they remain bounded. 

In contrast, in the space-homogeneous setting, where the initial data depends only on the velocity variable $v$, the well-posedness theory is considerably more complete. Global-in-time existence of $\mathcal H$-solutions has been established in \cite{Villani1998LandauHsolutions,desvill_villani2000landau_hard,AlexLiaoLin2015APrioriLandauSoft,wu2014GlobalEstimLandau,fournier2009well_posedness_soft_potentials}. The formation of singularities was ruled out in \cite{GS24}, where the Fisher information is shown to be monotone decreasing. This result has since been extended to increasingly rough initial data in \cite{GualdaGoldLohDev2024LandauFisherL1,Ji2024LandauDissFisher,GualdGoldLoher2025LandauL32,HeJiLuo2025LandauH-12}.

\textbf{Delocalized kinetic equations and the fuzzy Landau equation.} The notion of nonlocal collisions in kinetic equations was first introduced by Morgenstern \cite{Morgen1955MaxwellBoltzmann} in the context of the Maxwell--Boltzmann equation. Several related models were subsequently developed, including the Povzner equation \cite{povzner1962boltzmann} and the Enskog equation \cite{Arkeryd1986Enskog2D,Arkeryd1989Enskog,toscaniBellomo1987Enskog}. On the other hand, fuzzy kinetic models have been introduced more recently. These were first proposed for the Boltzmann equation in \cite{Erbar2025FuzzyBoltzmann,erbar2025LimitFuzzyBoltzmann}, and later extended to the Landau equation in the series \cite{duong2025FuzzyLandau1,duong2025FuzzyLandau2,duong2025FuzzyLandau3}. In this setting, the authors develop a variational formulation of \eqref{eq:fuzzy_Landau}, establish global existence of $\mathcal{H}$-solutions, and derive \eqref{eq:fuzzy_Landau} as the grazing-collision limit of the fuzzy Boltzmann equation, thereby extending the classical results of \cite{AlexVillani2004LandauApproxi}. Shortly after the introduction of \cref{eq:fuzzy_Landau}, global existence and uniqueness of smooth solutions for regular initial data were established in \cite{gualdani2025FuzzyLandauFisher} for $-\sqrt 7 \leq \gamma \leq 0$.

\textbf{The multispecies Landau equation.}
The literature on the multiespecies Landau equation \eqref{eq:multiespecies_Landau} remains limited. It was first analyzed in the inhomogeneous setting in \cite{GualdaniZamponi2017MultiespLandau}, where a spectral gap for the linearized operator was established, together with existence and uniqueness of solutions for smooth initial data close to equilibrium and exponential convergence to the steady state. In a related perturbative framework, \cite{YuHaitaoKung2018LinearMultiLandau} studied a two-species system in which one species is near vacuum and the other near equilibrium, obtaining instantaneous smoothing for the linearized equation.

For general initial data, existence results are currently available only in the homogeneous setting. In the recent preprint \cite{JunnWinterYoldas2025multiLandau}, global existence of smooth solutions is established for interaction potentials in the range $-\sqrt{8}\leq \gamma<-2$, by extending the symmetrization method of \cite{GS24} and proving monotonicity of the Fisher information. Independently, monotonicity of the Fisher information was obtained in \cite{zhu2025fisherinformationmultispecieslandau} using a similar approach. These results are an important motivation for our second main theorem.

\subsection{Main results}
Our main contribution is to extend the uniqueness results of \cite{fournier2009well_posedness_soft_potentials,fournier2010uniqueness_Coulomb} to the fuzzy Landau equation \cref{eq:fuzzy_Landau} and the multiespecies Landau equation \cref{eq:multiespecies_Landau}, establishing explicit stability bounds under suitable integrability assumptions. The main difficulty lies in treating weak solutions without regularity assumptions, thereby accommodating very rough initial data. We also present a novel analytic re-interpretation of the stochastic argument from \cite{fournier2009well_posedness_soft_potentials,fournier2010uniqueness_Coulomb}.

Our stability estimates rely on the $2$-Wasserstein distance, which for probability densities $f,g$ on $\mathbb R^n$ is defined by
\begin{equation*}
d_2^2(f,g) := \inf_{\Pi\in\Gamma(f,g)} \int_{\mathbb R^n\times\mathbb R^n} |x-y|^2\,d\Pi(x,y),
\end{equation*}
where $\Gamma(f,g)$ denotes the set of transport plans, i.e., probability measures on $\mathbb{R}^n\times\mathbb{R}^n$ with marginals $f$ and $g$. This distance metrizes the space $\mathcal P_2(\mathbb R^n)$ of probability measures with finite second moments, endowed with the topology of weak convergence together with convergence of second moments. See \cite{villani2003topicsOT} for background on optimal transport.

We use a weak formulation of \cref{eq:fuzzy_Landau,eq:multiespecies_Landau}, in which all derivatives are transferred onto test functions. To this end, we introduce
\begin{equation*}
b(z):=\textnormal{div}_z \Phi(z) = -2 |z|^\gamma z.
\end{equation*}
We begin with the fuzzy Landau equation.
\begin{definition}\label{defn:weak_sol_fuzzy}
  Let $-3\leq \gamma \leq -2$, $q>\frac{3}{4+\gamma}$, and assume that $\kappa$ is bounded. Consider a curve of probability densities $t\in [0,T]\mapsto f(t)\in \mathcal P_2(\mathbb R^6)$ with $m_2(f):=\sup_{[0,T]}\int (|x|^2+|v|^2)f(t)<\infty$ and such that $f\in L^1_t L^q_v L^1_x$. We say that $f$ is a weak solution of the fuzzy Landau equation \cref{eq:fuzzy_Landau} if
  \begin{align*}
    \int_{\mathbb{R}^{6}}\varphi f(t)\,dx\,dv &
    = \int_{\mathbb{R}^{6}}\varphi f(0)\,dx\,dv	\\
    &+\int_0^t\int_{\mathbb{R}^{6}}\left[-\left(v\cdot\nabla_{x}\varphi\right)+2(\kappa b*f)\cdot\nabla_{v}\varphi+\left(\kappa\Phi*f\right):D_{v}^{2}\varphi\right]f\,dx\,dv\,ds,
  \end{align*}
  for every $\varphi\in\mathcal{C}_{b}^{2}(\mathbb{R}^{6})$ and every $t\in [0,T]$.
\end{definition}
%
We now state our first main result.
\begin{theorem}\label{thm:fuzzy}
  Consider $f,g$ weak solutions of \eqref{eq:fuzzy_Landau} for $-3\leq \gamma \leq -2$, and let $\sqrt{\kappa}$ be bounded and Lipschitz. Assume that $f,g\in L^1_t L^p_v L^1_x$ for $p=\frac{3}{3+\gamma}\in [3,\infty]$. Then, we have
  \begin{equation*}
    \sup_{0\leq t\leq T}d_2^2\big(f(t),g(t)\big) \leq \omega \left(d_2^2\big(f(0),g(0)\big)\right),
  \end{equation*}
  where $\omega:[0,\infty)\longrightarrow[0,\infty)$ is continuous, increasing, and satisfies $\omega(0)=0$. This modulus of continuity $\omega$ depends on $T$, $\Vert f\Vert_{L^1_tL^p_vL^1_x}$, $\Vert g\Vert_{L^1_tL^p_vL^1_x}$, and $\Vert\sqrt{\kappa}\Vert_{\mathrm{Lip}}$.
\end{theorem}
%
In particular, this establishes uniqueness within $L^1_t L^p_v L^1_x$. See outline of the proof for an explicit expression of $\omega$. We next extend the discussion to the multispecies Landau equation.
\begin{definition}\label{defn:weak_sol_multi}
  Let $-3\leq \gamma \leq 0$. Consider curves ${t\in [0,T]\mapsto f^{(i)}(t)\in \mathcal P_2(\mathbb R^3)}$ of probability densities, with $m_2(f^{(i)}):=\sup_{[0,T]}\int |v^i|^2f^{(i)}(t)<\infty$ for $i=1,\ldots,N$. Further assume that, in the case $-3\leq \gamma <-1$,  we have $f^{(i)}\in L^1_tL^q_{v^i}$ for some $q>\frac{3}{4+\gamma}$. We say that $(f^{(i)})_{i=1}^N$ is a weak solution of the multiespecies homogeneous Landau equation \cref{eq:multiespecies_Landau} if
  \begin{align*}
    \int_{\mathbb{R}^{3}}\varphi f^{(i)}(t)\,dv^{i}& 
    = \int_{\mathbb{R}^{3}}\varphi f^{(i)}(0)\,dv^{i}\\
    &
    +\sum_{j=1}^{N}\frac{c_{ij}}{m_{i}}\int_{0}^{t}\int_{\mathbb{R}^{3}}\left[\left(\frac{1}{m_{i}}\Phi*f^{(j)}\right):D_{v^{i}}^{2}\varphi+\left(\frac{1}{m_{i}}
    +\frac{1}{m_{j}}\right)(b*f^{(j)})\cdot\nabla_{v^{i}}\varphi\right]f^{(i)}\,dv^{i}\,ds,
  \end{align*}
  for every $\varphi\in\mathcal{C}_{b}^{2}(\mathbb{R}^{3})$ and every $t\in [0,T]$.
\end{definition}
Next, we present our second main result on the multispecies homogeneous Landau equation.
\begin{theorem}\label{thm:multiespecies}
  Let $(f^{(i)})_{i=1}^{N}$, $(g^{(i)})_{i=1}^{N}$ be weak solutions of \eqref{eq:multiespecies_Landau} for $-3\leq \gamma\leq 0$. Assume that $f^{(i)},g^{(i)}\in L^1_tL^p_{v^i}$ where $p=\frac{3}{3+\gamma}\in [1,\infty]$. Then there exists $\omega:[0,\infty)\longrightarrow[0,\infty)$ continuous, increasing, with $\omega(0)=0$ such that
  \begin{equation*}
    \sup_{0\leq t\leq T} \sum_{i=1}^{N}d_{2}^{2}\big(f^{(i)}(t),g^{(i)}(t)\big) \leq
    \omega\left(
      \sum_{i=1}^{N}d_{2}^{2}\big(f^{(i)}(0),g^{(i)}(0)\big)
    \right).
  \end{equation*}
  The modulus of continuity $\omega$ depends on $T$, $\gamma$, $N$, $\max_{ij}|c_{ij}|$, $\max_{i}|m_{i}|^{-1}$, $\max_i m_2(f^{(i)})$, and $\Vert f^{(i)}\Vert_{L^1_tL^p_{v^i}}$, $\Vert g^{(i)}\Vert_{L^1_tL^p_{v^i}}$. 
\end{theorem}
See the outline of the proof for an explicit expression of $\omega$.

\textbf{Comparison with existing results.} \cref{thm:fuzzy,thm:multiespecies} extend the conditional uniqueness results established in \cite{fournier2009well_posedness_soft_potentials,fournier2010uniqueness_Coulomb} for the homogeneous Landau equation. We note that in the non-Coulomb case, our result is slightly sharper, since we are able to reach the critical exponent ${p=\frac{3}{3+\gamma}}$, whereas \cite{fournier2009well_posedness_soft_potentials} shows uniqueness in $L^1_tL^p_v$ with $p>\frac{3}{3+\gamma}$. Finally, we note the related work \cite{FriesenMartinSundar2022BoltEnskog}, which establishes an analogue of  \cref{thm:fuzzy} for a delocalized version of the Boltzmann equation.

Regarding the fuzzy Landau equation, global existence and uniqueness of solutions with smooth initial data was established in \cite{gualdani2025FuzzyLandauFisher} for moderately soft potentials $\gamma \in [-\sqrt 7,0]$. In contrast, \cref{thm:fuzzy} accommodates highly irregular initial data and covers the regime of very soft potentials $\gamma \in [-3,-2]$. In relation to the multiespecies equation, well posedness of smooth solutions was shown in \cite{JunnWinterYoldas2025multiLandau} in the range $-\sqrt{8} \leq \gamma \leq 0$. On the other hand, \cref{thm:multiespecies} establishes uniqueness of weak solutions within $L^1_t L^p_{v^i}$ and applies to the full range of soft potentials $\gamma \in [-3,0]$.

It is worth noting that analogous conditional uniqueness results have been established for several well-known equations, including the $2$D Euler equation in vorticity form \cite{Yudovich1963Uniqueness}, the Vlasov–Poisson system \cite{Loeper2006UniquenessVlasovPoisson}, and the Patlak–Keller–Segel model \cite{carrillo2012uniquenessKS}. In fact, \cref{thm:fuzzy,thm:multiespecies} follow as direct consequences of a general conditional uniqueness result, which recovers these examples as particular cases within a unified framework. We refer to the outline below for further details.

\subsection{Outline of the proof}
We present two approaches for estimating the 2-Wasserstein distance in \cref{eq:fuzzy_Landau,eq:multiespecies_Landau}. The first extends the method of \cite{fournier2010uniqueness_Coulomb} to a broader class of nonlinear equations with singular coefficients, reinterpreting the original stochastic argument in a more analytic language. The second employs the symmetrization technique introduced in \cite{GS24} for the study of Fisher information.

\textbf{First approach: Coupling method.} The key observation is that \cref{eq:fuzzy_Landau,eq:multiespecies_Landau} share a common structure.  More precisely, both can be written as
\begin{equation}\label{eq:gen_eq}
  \partial_{t}f=\textnormal{div}\left(cf+(b*f)f\right)+D^{2}:\left(Mf+(A*f)f\right),
\end{equation}
where $c,b:\mathbb{R}^{n}\longrightarrow\mathbb{R}^{n}$, $A,M:\mathbb{R}^{n}\longrightarrow\mathbb{R}^{n\times n}$ and we denote ${D^2:(Mf) = \sum_{i,j} \partial_{x_ix_j}(M^{ij}f)}$. We prove \cref{thm:fuzzy,thm:multiespecies} by showing a uniqueness criterion for the more general equation \cref{eq:gen_eq} under the following assumptions on the coefficients.
\begin{enumerate}[label=(A\arabic*),ref=A\arabic*]
  \item \label{ass:PSD} The matrices $A$, $M$ can be expressed as ${A=\sigma\sigma^{T}}$, ${M=\mathfrak m \mathfrak m^{T}}$ for some $\sigma,\mathfrak m:\mathbb{R}^{n}\longrightarrow\mathbb{R}^{n\times n}$.
  \item \label{ass:Lipschitz} Both $c$ and $\mathfrak m$ are globally Lipschitz.
  \item \label{ass:growth} The variable $x$ can be decomposed as $x=(x^{1},\ldots,x^{k},x')\in(\mathbb{R}^{d})^{k}\times\mathbb{R}^{n-dk}$ so that
    \[
      |b(x)|\leq K\left(1+\sum_{i=1}^{k}|x^{i}|^{-\alpha+1}\right),\qquad
      |\sigma(x)|\leq K\left(1+\sum_{i=1}^{k}|x^{i}|^{-\frac{\alpha}{2}+1}\right),
    \]
    where $K>0$ and $0\leq \alpha \leq d$.
    \item \label{ass:local_Lipchitz} The coefficients $b$, $\sigma$ satisfy the following local Lipschitz estimate away from $x^{i}=0$:
    \begin{equation*}
      \frac{|b(x)-b(y)|}{|x-y|} + \frac{|\sigma(x)-\sigma(y)|^2}{|x-y|^2}\leq K\left(1+\sum_{i=1}^{k}\left(|x^{i}|^{-\alpha}+|y^{i}|^{-\alpha}\right)\right),
    \end{equation*}
\end{enumerate}
Our computations will require that the $x^i$ marginals of $f$ satisfy an integrability condition. We denote such marginals as
\begin{equation*}
  f^{(i)}:\mathbb{R}^{d}\longrightarrow\mathbb{R},\qquad f^{(i)}(x^{i}):=\int_{\mathbb{R}^{n-d}}f(x)\,dx^{-i},
\end{equation*}
where $dx^{-i}$ denotes integration with respect to all variables except $x^{i}$. The notion of weak solution to \eqref{eq:gen_eq} is a simple generalization of Definition \ref{defn:weak_sol_fuzzy}.
\begin{definition}
  Let $0\leq \alpha \leq d$. Consider a curve of probability densities $t\in [0,T]\mapsto f(t)\in \mathcal P_2(\mathbb R^n)$ with $m_2(f):=\sup_{[0,T]}\int|x|^2f(t)<\infty$. Further assume that, in the case $1 < \alpha \leq d$, we have $f^{(i)}\in L^1_tL^q_{x^i}$ for some $q>\frac{d}{d+1-\alpha}$. We say that $f$ is a weak solution of \eqref{eq:gen_eq} if 
  \begin{equation*}
    \int_{\mathbb{R}^{n}}\varphi f(t)\,dx=\int_{\mathbb{R}^{n}}\varphi f(0)\,dx+\int_{0}^{t}\int_{\mathbb{R}^{n}}\left[\Big(c+(b*f)\Big)\cdot\nabla\varphi+\Big(M+(A*f)\Big):D^{2}\varphi\right]\,dx\,ds,
  \end{equation*}
  for any $0\leq t\leq T$ and $\varphi\in \mathcal C^2_b$.
\end{definition}
Now, given two solutions $f(t,x)$, $g(t,y)$ of \eqref{eq:gen_eq}, we estimate $d_2^2(f(t),g(t))$ by constructing a curve of transport plans $\Pi(t)\in \Gamma(f(t),g(t))$, generated by a suitably chosen parabolic equation:
\begin{equation}\label{eq:plan_eq}
  \begin{cases}
     \partial_t\Pi = \textnormal{div}_{x,y}\left(\Big(\overline{c}+(\overline{b}*\Pi^{*})\Big)\Pi\right) + 
  D^2_{x,y}:\left(\Big(\overline{M}+(\overline{A}*\Pi^{*})\Big)\Pi\right),
  \\ 
  \Pi(0) = \Pi^*(0),
  \end{cases}
\end{equation}
where $\Pi^*(t)\in \Gamma(f(t),g(t))$ denotes the optimal transport plan and we introduced
\begin{align}
    \overline{c}(x,y)=\begin{pmatrix}c(x)\\c(y)\end{pmatrix}, &
    \quad\overline{M}(x,y)=\begin{pmatrix}\mathfrak m(x)\mathfrak m^{T}(x) & \mathfrak m(x)\mathfrak m^{T}(y)\\
    \mathfrak m(y)\mathfrak m^{T}(x) & \mathfrak m(y)\mathfrak m^{T}(y)\end{pmatrix},	\label{eq:bar_der}\\
    \overline{b}(x,y)=\begin{pmatrix}b(x)\\b(y)\end{pmatrix}, &
    \quad\overline{A}(x,y)=\begin{pmatrix}\sigma(x)\sigma^{T}(x) & \sigma(x)\sigma^{T}(y)\\
    \sigma(y)\sigma^{T}(x) & \sigma(y)\sigma^{T}(y)\end{pmatrix}.\label{eq:bar_der2}
\end{align}
Once the existence of a measure valued solution to \cref{eq:plan_eq} has been established, we take the test function ${\varphi(x,y)=\frac{1}{2}|x-y|^2}$, which yields
\begin{align*}
  \frac{1}{2}\frac{d}{dt}\int_{\mathbb R^{2n}}|x-y|^{2}\,d\Pi(t)
  &\leq C\int_{\mathbb R^{2n}}|x-y|^{2}\,d\Pi(t)
  +\int_{\mathbb R^{4n}}\left|\sigma(x-x_{*})-\sigma(y-y_{*})\right|^{2}\,d\Pi^{*}(t)\,d\Pi(t)\\
  &+\int_{\mathbb R^{4n}}|x-y|\left|b(x-x_{*})-b(y-y_{*})\right|\,d\Pi^{*}(t)\,d\Pi(t).
\end{align*}
Then the last two terms are estimated by splitting the integral into regions according to the size of $|x^i-x^i_*|$, $|y^i-y^i_*|$ and using assumptions \ref{ass:growth}, \ref{ass:local_Lipchitz}. In the end, we essentially show
\begin{equation*}
  \frac{d}{dt}d_2^2(f(t),g(t)) \leq C
  \max_{1\leq i\leq k}(\Vert f^{(i)}(t)\Vert_{L^p_{x^i}},\Vert g^{(i)}(t)\Vert_{L^p_{y^i}},m_2(f),m_2(g),1)
  \Psi\big(d_2^2(f(t),g(t))\big),
\end{equation*}
where $p=\frac{d}{d-\alpha}\in [1,\infty]$; $C$ depends on $\alpha$, $d$ and the constants from assumptions \ref{ass:Lipschitz}-\ref{ass:local_Lipchitz}; and
\begin{equation*}
  \Psi:[0,+\infty)\longrightarrow[0,+\infty),\qquad  \Psi(x):=\max(-x\log x, x) = 
  \begin{cases}
    -x\log x, & x<e^{-1},\\ 
    x, & x\geq e^{-1}.
  \end{cases}
\end{equation*}
This implies uniqueness due to Osgood's criterion. Moreover, we can explicitly solve the differential inequality to obtain the bound
\begin{equation}\label{eq:gen_bound}
  d_2^2(f(t),g(t)) \leq H\left(d_2^2(f(0),g(0)),
    C \int_{0}^{t}\max_{i}(\Vert f^{(i)}(t)\Vert_{L^p_{x^i}},\Vert g^{(i)}(t)\Vert_{L^p_{y^i}},m_2(f),m_2(g),1) \,ds
  \right),
\end{equation}
where $H(x,y)\geq 0$ is the continuous function with $H(0,\cdot)\equiv 0$ given by
\begin{equation}\label{eq:Omega}
  H(x,y):=
  \left\{ \begin{array}{ccc}
    x^{e^{-y}}, & 0\leq x\leq e^{-1}, & x\leq e^{-e^{y}},\\
    \frac{e^{y-1}}{-\ln x}, & 0\leq x\leq e^{-1}, & x\geq e^{-e^{y}},\\
    xe^{y}, & e^{-1}\leq x, & x\geq e^{y-1},\\
    e^{-\frac{1}{x}e^{-y-1}}, & e^{-1}\leq x, & x\leq e^{y-1}.
  \end{array}\right.
\end{equation}
In particular, if we choose the modulus of continuity
\begin{equation}\label{eq:mod_cont}
  \omega(x):=H\Big(x,
  C\max_{i}\big\{\Vert f^{(i)}(t)\Vert_{L^1_tL^p_{x^i}},\Vert g^{(i)}(t)\Vert_{L^1_tL^p_{x^i}},m_2(f),m_2(g),T\big\}\Big),
\end{equation}
then we have the following general result.
\begin{theorem}\label{thm:gen_thm}
  Consider $0\leq \alpha \leq d$. Let $f$, $g$ be weak solutions to \eqref{eq:gen_eq} with $f^{(i)},g^{(i)}\in L^1_tL^p_{x^i}$ for $p=\frac{d}{d-\alpha}\in [1,\infty]$. Then we have
  \begin{equation*}
    \sup_{0\leq t\leq T} d_2^2\big(f(t),g(t)\big) \leq \omega\left(d_2^2\big(f(0),g(0)\big)\right),
  \end{equation*}
  where the modulus of continuity $\omega$ is defined in \cref{eq:mod_cont}.
\end{theorem}
Note that when $b$ and $\sigma$ are globally Lipschitz we obtain unconditional uniqueness of \cref{eq:gen_eq}. This is a particular case of a more general result on nonlinear Fokker-Planck equations. See, for instance, \cite{Feng2018LandauTypeEqs} for a general statement. 

\textbf{Application to fuzzy Landau and multiespecies Landau.} \cref{thm:fuzzy,thm:multiespecies} are obtained as simple corollaries of the previous result. For the fuzzy Landau equation we have $n=6$, $d=3$, $k=1$ and replace $(x^1,x')\in \mathbb R^3\times \mathbb R^3$ by $(v,x)$. Moreover, the role of $c$ is taken by $(-v,\mathbf{0})$, we set $M\equiv \mathbf 0$, $b$ becomes $(\mathbf{0},-4\kappa (x)v|v|^{-\gamma})$, and $A$ is replaced by a $6\times 6$ matrix with all zeroes except for the block $\kappa(x)|v|^{2+\gamma}\mathbb{P}(v)$ corresponding to the operator $D^2_v$. For the choice of $\sigma$, it is convenient to note that $\mathbb P(z) = r(z)r(z)^\perp$, where
\begin{equation*}
  r(z):=|z|^{1+\frac{\gamma}{2}}
  \begin{pmatrix}z_{2} & -z_{3} & 0\\
    -z_{1} & 0 & z_{3}\\
    0 & z_{1} & -z_{2}
  \end{pmatrix}.
\end{equation*}
Then, it is not hard to see that in this particular example
\begin{align*}
  \left|\sigma(x,v)-\sigma(y,w)\right|&
  \leq C_{\gamma}(\Vert\sqrt{\kappa}\Vert_{\textnormal{Lip}}+1)\left(|x-y|+|v-w|\right)\left(|v|^{\frac{\gamma}{2}}+|v|^{\frac{\gamma}{2}+1}+|w|^{\frac{\gamma}{2}}+|w|^{\frac{\gamma}{2}+1}\right),\\
  \left|b(x,v)- b(y,w)\right|&
  \leq C_{\gamma}(\Vert\kappa\Vert_{\textnormal{Lip}}+1)\left(|x-y|+|v-w|\right)\left(|v|^{\gamma}+|v|^{\gamma+1}+|w|^{\gamma}+|w|^{\gamma+1}\right).
\end{align*}
So we verify assumptions \ref{ass:PSD}-\ref{ass:local_Lipchitz} when $-3\leq \gamma \leq -2$ with $\alpha = -\gamma$.

For the multiespecies Landau model we consider the equation satisfied by the tensor product $F:=f^{(1)}(t,v^1)\cdots f^{(N)}(t,v^N)$, which takes the form \eqref{eq:gen_eq} with $c \equiv \mathbf{0}$, $M \equiv \mathbf{0}$, and $b$, $A$ given by
\begin{equation*}
  -4\sum_{j=1}^N
  \begin{pmatrix}
    \frac{c_{1j}}{m_{j}}\left(\frac{1}{m_{1}}+\frac{1}{m_{j}}\right)v^{j}|v^{j}|^{-\gamma}\\
    \vdots\\
    \frac{c_{Nj}}{m_{j}}\left(\frac{1}{m_{1}}+\frac{1}{m_{j}}\right)v^{j}|v^{j}|^{-\gamma}
  \end{pmatrix},
  \quad 
  \sum_{j=1}^N
  \begin{pmatrix}\frac{c_{1j}}{m_{j}m_{1}}|v^{j}|^{2+\gamma}\mathbb{P}(v^{j}) & \mathbf{0} & \cdots & \mathbf{0}\\
  \mathbf{0} & \ddots &  & \vdots\\
  \vdots &  & \ddots & \mathbf{0}\\
  \mathbf{0} & \cdots & \mathbf{0} & \frac{c_{Nj}}{m_{j}m_{N}}|v^{j}|^{2+\gamma}\mathbb{P}(v^{j})
  \end{pmatrix}.
\end{equation*}
The verification of the hypothesis is a simple calculus exercise. Finally, it is not hard to see that if $\big(f^{(i)}\big)_{i=1}^N$ is a weak solution of the fuzzy Landau equation then $F=\bigotimes_{i=1}^Nf^{(i)}$ is a weak solution of \cref{eq:gen_eq}, so the application of \cref{thm:gen_thm} is justified.

\textbf{Other models included in the previous framework.} First we note that \cref{thm:gen_thm} can be extended to a broader class of equations. For instance, one may allow the coefficients $c$, $M$, $b$, and $A$ to depend on time, replace the convolution terms by more general interactions of the form $\int_{\mathbb{R}^d} b(x,x_*) f(x_*)\,dx_*$ and $\int_{\mathbb{R}^d} A(x,x_*) f(x_*)\,dx_*$, or consider the equation on a domain with boundary conditions that preserve mass. This allows us to recover a variety of models as particular cases of \cref{eq:gen_eq}. Let us briefly list some relevant examples.

\textit{The 2D incompressible Euler in vorticity form} models the evolution of the vorticity of a two dimensional inviscid fluid. The equation in the whole plane takes the form
\begin{equation}\label{eq:2D_Euler}
  \partial_{t}\omega_{t}+\textnormal{div}(u_{t}\omega_{t})=0,
\end{equation}
where $u_t$ can be expressed using Biot-Savart law as
\begin{equation*}
    u_{t}(x)=\int_{\mathbb R^2} K_{\textnormal{BS}}(x-x_{*})\omega(x_{*})\,dx_{*},\qquad 
    K_{\textnormal{BS}}(x)=\frac{1}{2\pi|x|^{2}}\begin{pmatrix}-x_{2}\\x_{1}\end{pmatrix}.
\end{equation*}
Therefore, if we assume that vorticity is nonnegative, bounded and integrable then \cref{thm:gen_thm} shows uniqueness within the class of solutions satisfying these properties. Uniqueness of bounded solutions to \cref{eq:2D_Euler} is a classical result which was proven in \cite{Yudovich1963Uniqueness} using different methods.

\textit{The Vlasov-Poisson equation} models the evolution of a distribution of charged particles which interact through the Coulomb force by neglecting collisional effects. The equation takes the form
\begin{equation*}
  \partial_{t}f_{t}+v\cdot\nabla_{x}f+E\cdot\nabla_{v}f=0.
\end{equation*}
When considering the whole $\mathbb R^3$, the electromagnetic force can be expressed as
\begin{equation*}
  E(x)=-\frac{1}{4\pi}\int_{\mathbb R^3}\frac{x-x_{*}}{|x-x_{*}|^{3}}\rho(x_{*})\,dx_{*},\qquad
  \rho(x) := \int_{\mathbb R^3}f(x,v)\,dv.
\end{equation*}
We note that this has exactly the same form as the drift term in the Landau equation. In this setting, uniqueness of bounded solutions was shown in \cite{Loeper2006UniquenessVlasovPoisson} also using optimal transportation techniques. In fact, the method we present can be seen as an extension of the one used in \cite{Loeper2006UniquenessVlasovPoisson}. In addition, \cref{thm:gen_thm} also applies to the two species Vlasov-Poisson system
\begin{equation*}
  \begin{cases}
    \partial_{t}f^{+}+v\cdot\nabla_{x}f^{+}+E\cdot\nabla_{v}f^{+}=0,\\
    \partial_{t}f^{-}+v\cdot\nabla_{x}f^{-}-E\cdot\nabla_{v}f^{-}=0,\\
    E=\displaystyle{-\frac{1}{4\pi}\int_{\mathbb R^3}\frac{x-x_{*}}{|x-x_{*}|^{3}}\left(\rho^{+}(x_{*})-\rho^{-}(x_{*})\right)\,dx_{*}},
  \end{cases}\quad 
  \rho^{\pm}(x):=\int_{\mathbb R^3} f^{\pm}(x,v)\,dv.
\end{equation*}
which models the evolution of electrons and ions. To the best of our knowledge, the uniqueness of bounded solutions for the previous two species system has not appeared explicitly in the literature, although it definitely seems expected.

\textit{The Patlak-Keller-Segel equation} models the evolution of a population of cells which move due to diffusion and to a chemical attractant. It takes the form
\begin{equation*}
  \begin{cases}
  \partial_{t}f+\textnormal{div}\left(f\nabla u\right)=\Delta f,\\
  -\Delta u+\alpha u=f.
  \end{cases}
\end{equation*}
Note that $\nabla u$ can be written as $\nabla \Phi * f$, where $\Phi$ is the fundamental solution of $-\Delta u+\alpha u$. Since $\Phi$ behaves like the fundamental solution of $-\Delta$ near $0$, we fall within the assumptions \ref{ass:PSD}-\ref{ass:local_Lipchitz} and \cref{thm:gen_thm} applies. Uniqueness of bounded solutions for these type of equations was proven in \cite{carrillo2012uniquenessKS}, and the method we present may be viewed as an extension of the one used in \cite{carrillo2012uniquenessKS}. Similarly to the Vlasov-Poisson equation \cref{thm:gen_thm} would also imply uniqueness of bounded solutions for a multiespecies version of Patlak-Keller-Segel and to Patlak-Keller-Segel on bounded domains with a boundary condition preserving mass. To our knowledge, these extensions of the uniqueness result of \cite{carrillo2012uniquenessKS} have not appeared in the literature, although they seem expected.

\textbf{Second approach: Symmetrization method.} We build on the symmetrization technique recently introduced in \cite{GS24}, where Fisher information was shown to decrease for the homogeneous Landau equation, leading to a resolution of its well-posedness theory. The key insight is that the derivative of Fisher information can be reformulated in terms of a linear parabolic equation, referred to as the lifted equation. Notably, the $2$-Wasserstein distance admits an analogous lifting property, suggesting the natural question of whether it also exhibits monotonicity properties.

This idea was developed in \cite{paper1} in collaboration with Delgadino, Gualdani, and Taskovic. Monotonicity was not obtained due to the absence of a suitable functional inequality. Nonetheless, the method provides a new proof of the conditional uniqueness criterion from \cite{fournier2009well_posedness_soft_potentials,fournier2010uniqueness_Coulomb}. In Section \ref{sec:symmetrization} we adapt the approach to the fuzzy Landau and the multiespecies Landau equations. Let us now outline the strategy in the case of the homogeneous Landau equation as presented in \cite{paper1}, which also captures the main ideas we use for \cref{eq:fuzzy_Landau,eq:multiespecies_Landau}.

Consider two solutions $f(t,v),g(t,v)$ of the homogeneous Landau equation. The main observation is that the time derivative of the $2$-Wasserstein distance $d_2^2(f(t),g(t))$ may be expressed as the derivative of $d_2^2$ along the flow of a linear parabolic equation in double variables. More precisely,
\begin{equation*}
  \frac{d}{dt}d_2^2(f(t),g(t)) = \frac{1}{2}\left.\frac{d}{d\tau}\right|_{\tau = 0}d_2^2(\tilde F(\tau),\tilde Q(\tau)),
\end{equation*}
where $\tilde F$, $\tilde G$ are solutions to linear equations in $\mathbb R^6$
\begin{equation*}
  \begin{cases}
    \partial_\tau \tilde F = Q(\tilde F)\\
    \tilde F(0,v,v_*) = f(t,v)f(t,v_*)
  \end{cases},\qquad
  \begin{cases}
    \partial_\tau \tilde G = Q(\tilde G)\\
    \tilde G(0,v,v_*) = g(t,v)g(t,v_*)
  \end{cases},
\end{equation*}
where $Q$ is an appropriate elliptic operator on $\mathbb R^6$. Therefore, estimating the growth of $d_2^2(f(t),g(t))$ reduces to the study of a linear model. Then the approach to estimate $\frac{d}{dt}d_2^2(f(t),g(t))$ is to replicate the following computation for the heat equation taking into account the structure of $Q$.
\begin{proposition*}
  Let $f$ and $g$ be two nonnegative, unit-mass solutions in $\mathbb R^n$ of the heat equation $\partial_t f = \Delta f$, $\partial_t g = \Delta g$. Then  $d_2^2(f,g)$ is non-increasing with dissipation
  \begin{equation*}
    \frac{d}{dt} d_2^2(f,g) = -\int_0^1\int_{\mathbb{R}^n}\left|D^2u\right|^2\rho(x)\,dx\,ds,
  \end{equation*}
  where $u$ is an optimal solution of the Hamilton–Jacobi equation \eqref{eq:HJ_formulation}, and $\rho(s,x)$, $0\leq s \leq 1$, denotes the $2$-Wasserstein geodesic joining $f$ and $g$, characterized by \eqref{eq:CE_heat}.
\end{proposition*}
\begin{proof}
  We use the dual formulation of the $2$-Wasserstein distance in terms of the Hamilton–Jacobi equation
  \begin{equation}\label{eq:HJ_formulation}
    d_{2}^{2}(f,g)=\sup\left\{ \int_{\mathbb R^n} u(1,x)g(x)\,dx-\int_{\mathbb R^n} u(0,x)f(x)\,dx:\partial_{s}u+\frac{1}{2}\left|\nabla u\right|^{2}=0\right\}.
  \end{equation}
  It is well known that an optimal potential $u$ generates the $2$-Wasserstein geodesic
  $(\rho(s,x))_{0\leq s\leq1}$ connecting $f$ and $g$ through the continuity equation
  \begin{equation}\label{eq:CE_heat}
    \partial_s \rho + \textnormal{div}(\rho \nabla u) = 0,\qquad \rho(0,\cdot) = f,\; \rho(1,\cdot) = g.
  \end{equation}
  See \cite{villani2003topicsOT} for more details on this formulation. Differentiating $d_2^2(f,g)$ yields
  \begin{align*}
    \frac{d}{dt}d_2^2(f,g) &
    =\int_{\mathbb{R}^{n}}u(1)\Delta g\,dx-\int_{\mathbb{R}^{n}}u(0)\Delta f\,dx
    =\int_{0}^{1}\frac{d}{ds}\int_{\mathbb{R}^{n}}\Delta u\rho\,dx\,ds \\
    & =-\int_{0}^{1}\int_{\mathbb{R}^{n}}\left(\frac{1}{2}\Delta\left|\nabla u\right|^{2}-\nabla\Delta u\cdot\nabla u\right)\rho\,dx\, ds 
    =-\frac{1}{2}\int_{0}^{1}\int_{\mathbb{R}^{n}}\left|D^{2}u\right|^{2}\rho\,dx\,ds,
  \end{align*}
  where we have used Bochner's formula $\frac{1}{2}\Delta\left|\nabla u\right|^{2}=\nabla\Delta u\cdot\nabla u+\left|D^{2}u\right|^{2}$.
\end{proof}
We note that this approach does not cover the Coulomb case $\gamma = -3$. In addition, a rigorous justification of the method requires the function $u$ to possess three derivatives. Such regularity can only be ensured when $f$ and $g$ are of class $\mathcal{C}^1$ and locally bounded away from zero. These drawbacks led us to adopt the first approach that we previously described. Nevertheless, we include this symmetrization method as a formal alternative for deriving comparable estimates.

\textbf{Organization of the paper.} Section \ref{sec:gen_thm} is devoted to the proof of \cref{thm:gen_thm}. A distinction arises between the cases $2\leq \alpha\leq d$, and $0<\alpha <2$, as assumption \ref{ass:growth} changes qualitatively with the sign of the exponents. Sections \ref{subsec:main_comput}-\ref{subsec:existence_plans} address the case $2\leq \alpha\leq d$. The core idea of the proof is presented in Section \ref{subsec:main_comput}. Section \ref{subsec:coeff_estimates} develops the necessary estimates on the coefficients, while Section \ref{subsec:existence_plans} establishes the existence of the required curve of transport plans \cref{eq:plan_eq}. The modifications needed to treat the range $0 < \alpha < 2$ are described in Section \ref{subsec:soft_pot_modif}. Finally, Section \ref{sec:symmetrization} contains the method based on a symmetrization argument, with Section \ref{subsec:sym_fuzzy} devoted to the fuzzy Landau equation, and Section \ref{subsec:sym_multi} to the multiespecies homogeneous Landau equation.

\textbf{Notation.} While deriving inequalities, we often write a generic constant $C$ whose value may change from line to line. Some of these are differential inequalities with respect to $t$ which, in all rigor, should be understood in their integrated version. From here on, we use the shorthand $f_t=f(t,\cdot)$, and let $m_2(f)$ be an upper bound for the second moments of $f$. The $x^{i}$ marginal of $f$ is denoted as $f^{(i)}$. We use four sets of variables $x$, $x_*$, $y$, $y_*$, reserving $x,x_*$ for expressions involving $f$; $y,y_*$ for expressions involving $g$; $x,y$ for integration with respect to $\Pi_t$; and $x_*,y_*$ for integration with respect to $\Pi^*_t$. We denote the optimal transport plan between $f_t$, $g_t$ as $\Pi^*_t$. Finally, we denote the tensor product and direct sum of functions as $(\varphi\otimes \psi)(x,x_*)=\varphi(x)\psi(x_*)$, $(\varphi\oplus \psi)(x,x_*)=\varphi(x)+\psi(x_*)$.

\section{A general uniqueness result} \label{sec:gen_thm}

For the rest of the section, we consider $f,g$ weak solutions of \cref{eq:gen_eq}.

\subsection{Main computation}\label{subsec:main_comput}
In this section we present the main computations which prove \cref{thm:gen_thm} in the case $2\leq \alpha\leq d$. The modifications required to treat $0<\alpha <2$ are detailed in Section \ref{subsec:soft_pot_modif}. The following lemma, which formalizes \cref{eq:plan_eq} and holds for all $0\leq \alpha \leq d$, is the key ingredient in the proof.
\begin{lemma}\label{lem:existence_plan}
  There exists a curve of probability measures $t\mapsto \Pi_t$ such that $\Pi_t\in \Gamma(f_t,g_t)$ and such that for any $\varphi \in \mathcal C^2_b(\mathbb R^n\times \mathbb R^n)$ we have that $t\mapsto \int \varphi\,d\Pi_t$ is absolutely continuous with
  \begin{equation*}
    \frac{d}{dt}\int\varphi\,d\Pi_{t} = \int\Big(\left(\overline{c}+(\overline{b}*\Pi_{t}^{*})\right)\cdot\nabla_{x,y}\varphi 
    + \left(\overline{M}+(\overline{A}*\Pi_{t}^{*})\right):D_{x,y}^{2}\varphi\Big)\,d\Pi_{t},
  \end{equation*}
  where $\Pi^*_t\in \Gamma(f_t,g_t)$ is the optimal transport plan and $\overline c$, $\overline b$, $\overline M$, $\overline A$ were defined in \cref{eq:bar_der,eq:bar_der2}.
\end{lemma}
Note that we do not claim uniqueness of $\Pi_t$ nor is it used in what follows. We also require the following estimates, which generalize Lemmas 6 and 7 of \cite{fournier2010uniqueness_Coulomb}.

\begin{lemma}\label{lem:A_b_estimates}
  Given $2\leq \alpha \leq d$ and $p=\frac{d}{d-\alpha}\in [1,\infty]$, we have:
  \begin{align}
    \int\left|\sigma(x-z)-\sigma(y-z)\right|^{2}f_t(z)\,dz &
    \leq C\max_{1\leq i \leq k}(\Vert f_{t}^{(i)}\Vert _{L^{p}},1)\Psi(|x-y|^{2}),
    \label{eq:A_estimate}\\
    \int\left|b(x-z)-b(y-z)\right|f_{t}(z)\,dz & 
    \leq C\max_{1\leq i \leq k}(\Vert f_{t}^{(i)}\Vert _{L^{p}},1)\Psi(|x-y|),
    \label{eq:b_estimate}\\
    \int\left|b(x-x_{*})-b(y-y_{*})\right||x-y|\,d\Pi_{t}^{*}\,d\Pi_{t} &
    \leq C\max_{1\leq i \leq k}( \Vert f_{t}^{(i)}+g_{t}^{(i)}\Vert _{L^{p}},1)
    \Psi\left(\int|x-y|^{2}\,d\Pi_{t}\right)
    \label{eq:b_estimate2},
  \end{align}
  where $\Psi(x)=\max\left\{ x,-x\log x\right\}$ and $C$ depends on $n$, $d$, $k$ and the constants in assumptions \ref{ass:growth},~\ref{ass:local_Lipchitz}.
\end{lemma}

We now turn to the main part of the proof.

\begin{proof}[Proof of \cref{thm:gen_thm} assuming \cref{lem:existence_plan,lem:A_b_estimates}]
  We show a differential inequality of the form
  \begin{equation*}
    \frac{d}{dt}\int|x-y|^{2}\,d\Pi_{t}\leq C\max_{1\leq i \leq k}( \Vert f_{t}^{(i)}\Vert _{L^{p}},\Vert g_{t}^{(i)}\Vert _{L^{p}},1) \Psi\left(\int|x-y|^{2}\,d\Pi_{t}\right),
  \end{equation*}
  Where $C>0$ depends on $n$, $d$, $k$ and the constants in the assumptions. Using $\varphi(x,y)=\frac{1}{2}|x-y|^2$ as a test function in \cref{lem:existence_plan} we obtain
  \begin{align*}
    \frac{1}{2}\frac{d}{dt}\int|x-y|^{2}\,d\Pi_{t} &
    =\int\left[\left(c(x)-c(y)\right)\cdot(x-y)+\left|\mathfrak m(x)-\mathfrak m(y)\right|^{2}\right]\,d\Pi_{t}	\\
    & +\int\left[\left(b(x-x_{*})-b(y-y_{*})\right)\cdot(x-y)+\left|\sigma(x-x_{*})-\sigma(y-y_{*})\right|^{2}\right]\,d\Pi_{t}^{*}\,d\Pi_{t}.
  \end{align*}
  Assumption \ref{ass:Lipschitz} on $c$, $\mathfrak m$ allows us to bound the first term by a multiple of $\int|x-y|^{2}\,d\Pi_{t}$. The term involving $b$ has already been taken care of in \cref{lem:A_b_estimates}. It only remains to estimate
  \begin{equation*}
    I:=\int\left|\sigma(x-x_{*})-\sigma(y-y_{*})\right|^{2}\,d\Pi_{t}^{*}\,d\Pi_{t}.
  \end{equation*}
  We do this by simply applying a triangle inequality and using that $\Pi_t,\Pi_t^*\in \Gamma(f_t,g_t)$:
  \begin{align*}
    \frac{1}{2}I &\leq  \int\left|\sigma(x-x_{*})-\sigma(x-y_{*})\right|^{2}\,d\Pi_{t}^{*}\,d\Pi_{t}
    + \int\left|\sigma(x-y_{*})-\sigma(y-y_{*})\right|^{2}\,d\Pi_{t}^{*}\,d\Pi_{t}\\
    & =\int\left(\int\left|\sigma(x-x_{*})-\sigma(x-y_{*})\right|^{2}f(x)\,dx\right)\,d\Pi_{t}^{*}\\
    & +\int\left(\int\left|\sigma(x-y_{*})-\sigma(y-y_{*})\right|^{2}g(y_{*})\,dy\right)\,d\Pi.
  \end{align*}
  Now using \cref{lem:A_b_estimates}, the concavity of $\Psi$, and Jensen's inequality yields
  \begin{align*}
    I & \leq C\max_{1\leq i \leq k}(\Vert f_{t}^{(i)}\Vert _{L^{p}},\Vert g_{t}^{(i)}\Vert _{L^{p}},1)
    \left[\int\Psi(|x_{*}-y_{*}|^{2})\,d\Pi_{t}^{*}+\int\Psi(|x-y|^{2})\,d\Pi_{t}\right]\\
    &\leq C\max_{1\leq i \leq k}(\Vert f_{t}^{(i)}\Vert _{L^{p}},\Vert g_{t}^{(i)}\Vert _{L^{p}},1)\Psi\left(\int|x-y|^{2}\,d\Pi_{t}\right).
  \end{align*}
  This shows the desired differential inequality:
  \begin{equation*}
    \frac{d}{dt}\int|x-y|^{2}\,d\Pi_{t}\leq C\max_{1\leq i \leq k}( \Vert f_{t}^{(i)}\Vert _{L^{p}},\Vert g_{t}^{(i)}\Vert _{L^{p}},1) \Psi\left(\int|x-y|^{2}\,d\Pi_{t}\right),
  \end{equation*}
  which implies uniqueness due to Osgood's criterion. Moreover, we have the explicit bound
  \begin{align*}
    d_2^2(f(t),g(t))  \leq \int |x-y|^2\,d\Pi_t	
    & \leq H\left(\int |x-y|^2\,d\Pi_0,
    C \int_{0}^{t}\max_{i}(\Vert f^{(i)}(t)\Vert_{L^p_{x^i}},\Vert g^{(i)}(t)\Vert_{L^p_{y^i}},1) \,ds
  \right),\\
  & =H\left(d_2^2(f(0),g(0)),
    C \int_{0}^{t}\max_{i}(\Vert f^{(i)}(t)\Vert_{L^p_{x^i}},\Vert g^{(i)}(t)\Vert_{L^p_{y^i}},1) \,ds
  \right),
  \end{align*}
  where $H(x,y)\geq 0$ is the continuous function with $H(0,\cdot)\equiv 0$ defined in \cref{eq:Omega}.
\end{proof}

\subsection{Estimates involving the coefficients}\label{subsec:coeff_estimates}
In this section we prove \cref{lem:A_b_estimates} using the following elementary inequalities, which are a small generalization of Lemma 4 of \cite{fournier2010uniqueness_Coulomb}.
\begin{lemma}\label{lem:sing_ints}
  Let $0\leq \beta <\alpha \leq d$, $p=\frac{d}{d-\alpha}\in [1,\infty]$ and $h\in L^{p}(\mathbb{R}^{d})$. Then, given $0<\varepsilon<1$,
  \begin{align}
    \sup_{x}\int_{\mathbb{R}^{d}}|x-x_{*}|^{-\beta}h(x_{*})\,dx_{*}&
    \leq C_{d,\alpha,\beta}\left(1+\Vert h\Vert _{L^{p}}\right),\label{eq:sing_int}\\
    \sup_{x,y}\int_{|y-x_{*}|\leq\varepsilon}|x-x_{*}|^{-\beta}h(x_{*})\,dx_{*}&
    \le C_{d,\alpha,\beta}\varepsilon^{\alpha-\beta}\Vert h\Vert _{L^{p}},\label{eq:sing_int_eps}\\
    \sup_{x}\int_{|x-x_{*}|\geq\varepsilon}|x-x_{*}|^{-\alpha}h(x_{*})\,dx_{*}&
    \leq C_{d}\left(1-\log\varepsilon\right)\Vert h\Vert _{L^{p}}.\label{eq:sing_int_log}
  \end{align}
\end{lemma}
\begin{proof}
  For \eqref{eq:sing_int} we separate the integral into the regions where $|x-x_{*}|\leq1$, $|x-x_{*}|\geq1$ and use Hölder inequality with exponents $p=\frac{d}{d-\alpha}$, $\frac{d}{\alpha}$
  \begin{equation*}
    \int_{\mathbb{R}^{d}}|x-x_{*}|^{-\beta}h(x_{*})\,dx_{*}
    \leq\int_{|x-x_{*}|\le1}|x-x_{*}|^{-\beta}h(x_{*})\,dx_{*}+1
    \leq C_{d,\alpha,\beta}\left\Vert h\right\Vert _{L^{p}}+1,
  \end{equation*}
  where $C_{d,\alpha,\beta}\nearrow+\infty$ as $\beta \nearrow\alpha$. We also obtain \eqref{eq:sing_int_eps} from Hölder's inequality:
  \begin{align*}
    \int_{|y-x_*|\leq \varepsilon}|x-x_{*}|^{-\beta}h(x_{*})\,dx &
    \leq\left\Vert h\right\Vert _{L^{p}}\left(\int_{|z|\leq\varepsilon}|z|^{-\frac{\beta}{\alpha}d}\,dz\right)^{\frac{\alpha}{d}}	\\
    & \leq C_{d}\left\Vert h\right\Vert _{L^{p}}\left(\int_{0}^{\varepsilon}r^{\left(1-\frac{\beta}{\alpha}\right)d-1}\,dz\right)^{\frac{\alpha}{d}}
    =C_{d,\alpha,\beta}\left\Vert h\right\Vert _{L^{p}}\varepsilon^{\alpha-\beta}.
  \end{align*}
  Finally, for \eqref{eq:sing_int_log} we separate into regions where $|x-x_*|\geq 1$, $\varepsilon \leq |x-x_*|\leq 1$
  \begin{equation*}
    \int_{|x-x_{*}|\geq\varepsilon}|x-x_{*}|^{-\alpha}h(x_{*})\,dx\leq\int_{\varepsilon\leq|x-x_{*}|\leq1}|x-x_{*}|^{-\alpha}h(x_{*})\,dx+1,
  \end{equation*}
  and use Hölder inequality with the same exponents to bound the singular integral.
\end{proof}
We use the previous estimates on the marginals $f_t^{(i)}$, $g_t^{(i)}$ together with assumptions \ref{ass:growth}, \ref{ass:local_Lipchitz}.

\begin{proof}[Proof of estimate \cref{eq:A_estimate} from \cref{lem:A_b_estimates}]
  We must show that for $2\leq \alpha \leq d$ we have
  \begin{equation*}
    \int\left|\sigma(x-z)-\sigma(y-z)\right|^{2}f_t(z)\,dz 
    \leq C\max_{1\leq i \leq k}(\Vert f_{t}^{(i)}\Vert _{L^{p}},1)\Psi(|x-y|^{2}).
  \end{equation*}
  We split the integral into four regions depending on ${0<\varepsilon<1}$, a parameter that is chosen later. To simplify the expressions, we omit the inocuous $+1$ term from assumptions \ref{ass:growth} and \ref{ass:local_Lipchitz}.

  \textit{Region 1: $R_{1}:=\left\{ \min_{i}|z^{i}-x^{i}|\leq\varepsilon\right\}$.} By symmetry, we can reduce it to $R_{1}=\left\{ |z^{1}-x^{1}|\leq\varepsilon\right\}$, where we apply assumption assumption \ref{ass:growth}
  \begin{equation*}
    \int_{R_{1}}|\sigma(x-z)-\sigma(y-z)|^{2}f_{t}\,dz\leq C\sum_{i=1}^{k}\int\mathds{1}_{|x^{1}-z^{1}|\leq\varepsilon}\left(|x^{i}-z^{i}|^{-\alpha+2}+|y^{i}-z^{i}|^{-\alpha+2}\right)f_{t}\,dz.
  \end{equation*}
  Now use Hölder inequality twice with exponents $\left(\frac{d}{d-1},d\right)$, $(p=\frac{d}{d-\alpha},\frac{d}{\alpha})$ and apply \cref{eq:sing_int}:
  \begin{align*}
    \int\mathds{1}_{|x^{1}-z^{1}|\leq\varepsilon}|x^{i}-z^{i}|^{-\alpha+2}f_{t}\,dz &
    \leq\left(\int\mathds{1}_{|x^{1}-z^{1}|\leq\varepsilon}f_{t}^{(1)}\,dz^{1}\right)^{\frac{1}{d}}
    \left(\int|x^{i}-z^{i}|^{-\frac{d(\alpha-2)}{d-1}}f_{t}^{(i)}\, dz^i\right)^{1-\frac{1}{d}}	\\
    &\leq C\varepsilon^{\frac{\alpha}{d}}\max_{1\leq i\leq k}(\Vert f_{t}^{(i)}\Vert_{L^{p}},1).
  \end{align*}
  The same applies to the term involving $|y^i-z^i|^{-\alpha+2}$. Remembering that $\alpha \geq 2$, we obtain
  \begin{equation*}
    \int_{R_{1}}|\sigma(x-z)-\sigma(y-z)|^{2}f_{t}\,dz\leq C\varepsilon^{\frac{2}{d}}\max_{1\leq i\leq k}(\Vert f_{t}^{(i)}\Vert_{L^{p}},1).
  \end{equation*}
  \textit{Region 2: $R_{2}:=\left\{ \min_{i}|z^{i}-y^{i}|\leq\varepsilon\right\}$.} Is symmetric to the previous case. 

  \textit{Region 3: $R_{3}=\left\{ \min_{i}|z^{i}-x^{i}|\geq\varepsilon,\min_{i}|z^{i}-y^{i}|\geq\varepsilon\right\}$.} Here assumption \ref{ass:local_Lipchitz} and (\ref{eq:sing_int_log}) show
  \begin{align*}
    & \int_{R_{3}}|\sigma(x-z)-\sigma(y-z)|^{2}f_{t}\,dz \\
    \leq&\;  C|x-y|^{2} \sum_{i=1}^{k}\bigg[\int_{|z^{i}-x^{i}|\geq\varepsilon}|z^{i}-x^{i}|^{-\alpha}f_{t}^{(i)}\,dz^{i}
    +\int_{|z^{i}-y^{i}|\geq\varepsilon}|z^{i}-y^{i}|^{-\alpha}f_{t}^{(i)}\,dz^{i}\bigg] \\
    \leq &\;  C\left(1-\log\varepsilon\right)\max_{1\leq i \leq k}(\Vert f_{t}^{(i)}\Vert _{L^{p}},1)|x-y|^{2}.
  \end{align*}
  Putting all the regions together, we have
  \begin{equation*}
    \int\left|\sigma(x-z)-\sigma(y-z)\right|^{2}f_{t}(z)\,dz\leq C \max_{1\leq i \leq k}(\Vert f_{t}^{(i)}\Vert _{L^{p}},1)\left[\left(1-\log\varepsilon\right)|x-y|^{2}+\varepsilon^{\frac{2}{d}}\right].
  \end{equation*}
  Now we simply pick $\varepsilon^{\frac{2}{d}}=\min( e^{-1},|x-y|^2)$.
\end{proof}

The proof of estimate \cref{eq:b_estimate} from \cref{lem:A_b_estimates} follows the same argument, so we omit it. The last estimate is a bit more involved, so we present its proof separately.

\begin{proof}[Proof of estimate \cref{eq:b_estimate2} from \cref{lem:A_b_estimates}]
  We must show that for $\alpha \geq 2$ we have
  \begin{equation*}
    \int\left|b(x-x_{*})-b(y-y_{*})\right||x-y|\,d\Pi_{t}^{*}\,d\Pi_{t}
    \leq C\max_{1\leq i \leq k}( \Vert f_{t}^{(i)}\Vert _{L^{p}},\Vert g_{t}^{(i)}\Vert _{L^{p}},1)
    \Psi\left(\int|x-y|^{2}\,d\Pi_{t}\right).
  \end{equation*}
  To simplify notation, we define
  \begin{equation*}
    I(x,y,x_{*},y_{*}):=|x-y|\left|b(x-x_{*})-b(y-y_{*})\right|.
  \end{equation*}
  As before, we split the integral into regions depending on $0<\varepsilon<1$, which is suitably chosen later and omit the $+1$ term in the assumptions for simplicity.
  
  \textit{Region 1: $R_{1}:=\left\{ \min_{i}|x^{i}-x_{*}^{i}|\leq\varepsilon\right\}$.} By symmetry, we can reduce it to $R_{1}=\left\{ |x^{1}-x_{*}^{1}|\leq\varepsilon\right\}$, where assumption \ref{ass:growth} shows
  \begin{align}
    \int_{R_{1}}I\,d\Pi_{t}\,d\Pi_{t}^{*}
    & +C\sum_{i=1}^{k}\int|x-y|\left(\int\mathds{1}_{|x^{1}-x_{*}^{1}|<\varepsilon}|x^{i}-x_{*}^{i}|^{-\alpha+1}f_{t}\,dx_{*}\right)\,d\Pi_{t}\nonumber\\
    & +C\sum_{i=1}^{k}\int|x-y|\left(\int\mathds{1}_{|x^{1}-x_{*}^{1}|<\varepsilon}|y^{i}-y_{*}^{i}|^{-\alpha+1}\,d\Pi_{t}^{*}\right)\,d\Pi_{t}.\label{eq:term3_reg1}
  \end{align}
  Apply Hölder's inequality twice with exponents $\left(\frac{2d}{2d-1},2d\right)$, $\left(p=\frac{d}{d-\alpha},\frac{d}{\alpha}\right)$ and use \cref{eq:sing_int}:
  \begin{align*}
    \int\mathds{1}_{|x^{1}-x_{*}^{1}|<\varepsilon}|x^{i}-x_{*}^{i}|^{-\alpha+1}f_{t}\,dx_{*}
    &\leq\left(\int\mathds{1}_{|x^{1}-x_{*}^{1}|<\varepsilon}f_{t}^{(1)}\,dx_{*}^{1}\right)^{\frac{1}{2d}}
    \left(\int|x^{i}-x_{*}^{i}|^{-\frac{2d(\alpha-1)}{2d-1}}f_{t}^{(i)}\,dx_{*}^{i}\right)^{1-\frac{1}{2d}}\\
    &\leq C\varepsilon^{\frac{\alpha}{2d}}\max_{1\leq i\leq k}(\Vert f^{(i)}\Vert_{L^{p}},1).
  \end{align*}
  The third term \cref{eq:term3_reg1} is estimated the same way:
  %
  \begin{align*}
    \int\mathds{1}_{|x^{1}-x_{*}^{1}|<\varepsilon}|y^{i}-y_{*}^{i}|^{-\alpha+1}\,d\Pi_{t}^{*} 
    &\leq\left(\int\mathds{1}_{|x^{1}-x_{*}^{1}|<\varepsilon}f_{t}^{(1)}\,dx_{*}^{1}\right)^{\frac{1}{2d}}
    \left(\int|y^{i}-y_{*}^{i}|^{-\frac{2d(\alpha-1)}{2d-1}}g_{t}^{(i)}\,dy_{*}^{i}\right)^{1-\frac{1}{2d}}	\\
    & \leq C\varepsilon^{\frac{\alpha}{2d}}\max(\Vert f_{t}^{(1)}\Vert _{L^{p}},\Vert g_{t}^{(i)}\Vert _{L^{p}},1).
  \end{align*}
  Keeping in mind $\alpha \geq 2$, this yields
  \begin{equation*}
    \int_{R_{1}}I\,d\Pi_{t}\,d\Pi_{t}^{*}\leq 
    C\varepsilon^{\frac{1}{d}}\max_{1\leq i \leq k}(\Vert f_{t}^{(i)}\Vert _{L^{p}},\Vert g_{t}^{(i)}\Vert _{L^{p}},1)\left(\int|x-y|^2\,d\Pi_t\right)^\frac{1}{2}.
  \end{equation*}
  \textit{Region 2: $R_{2}:=\left\{ \min_{i}|y^{i}-y_{*}^{i}|\leq\varepsilon\right\}$.} Is symmetric to the previous case.

  \textit{Region 3: $R_{3}=\left\{ \min_{i}|x^{i}-x_{*}^{i}|\geq\varepsilon,\min_{i}|y^{i}-y_{*}^{i}|\geq\varepsilon\right\}$.} Here we use assumption \ref{ass:local_Lipchitz}, which gives
  \begin{equation*}
    I\leq C\left(|x-y|^{2}+|x_{*}-y_{*}|^{2}\right)\left(\sum_{i=1}^{k}|x^{i}-x_{*}^{i}|^{-\alpha}+\sum_{i=1}^{k}|y^{i}-y_{*}^{i}|^{-\alpha}\right),
  \end{equation*}
  yielding
  \begin{align*}
    \int_{R_{3}}I\,d\Pi_{t}\,d\Pi_{t}^{*}&
    \leq C\int|x-y|^{2}\left(\sum_{i=1}^{k}\int_{|x^{i}-x_{*}^{i}|\geq\varepsilon}|x^{i}-x_{*}^{i}|^{-\alpha}f_{t}^{(i)}\,dx_{*}^{i}\right)\,d\Pi_{t}\\
    &+C\int|x-y|^{2}\left(\sum_{i=1}^{k}\int_{|y^{i}-y_{*}^{i}|\geq\varepsilon}|y^{i}-y_{*}^{i}|^{-\alpha}g_{t}^{(i)}\,dy_{*}^{i}\right)\,d\Pi_{t}\\
    &+C\int|x_{*}-y_{*}|^{2}\left(\sum_{i=1}^{k}\int_{|x^{i}-x_{*}^{i}|\geq\varepsilon}|x^{i}-x_{*}^{i}|^{-\alpha}f_{t}^{(i)}\,dx^{i}\right)\,d\Pi_{t}^{*}\\
    &+C\int|x_{*}-y_{*}|^{2}\left(\sum_{i=1}^{k}\int_{|y^{i}-y_{*}^{i}|\geq\varepsilon}|y^{i}-y_{*}^{i}|^{-\alpha}g_{t}^{(i)}\,dy^{i}\right)\,d\Pi_{t}^{*}.
  \end{align*}
  Now \cref{eq:sing_int_log} shows
  \begin{equation*}
    \int_{R_{3}}I\,d\Pi_{t}\,d\Pi_{t}^{*}\leq 
    C(1-\log\varepsilon)\max_{1\leq i \leq k}(\Vert f_{t}^{(i)}\Vert _{L^{p}},\Vert g_{t}^{(i)}\Vert _{L^{p}},1)\int|x-y|^{2}\,d\Pi_{t}.
  \end{equation*}
  All in all, we have
  \begin{equation*}
    \int I\,d\Pi_{t}\,d\Pi_{t}^{*}\leq 
    C\max_{1\leq i \leq k}(\Vert f_{t}^{(i)}\Vert _{L^{p}},\Vert g_{t}^{(i)}\Vert _{L^{p}},1)\left((1-\log\varepsilon)\int|x-y|^{2}\,d\Pi_{t}+\varepsilon^{\frac{1}{d}}\left(\int|x-y|^2\,d\Pi_t\right)^\frac{1}{2}\right).
  \end{equation*}
  Now we may choose $\displaystyle\varepsilon^\frac{1}{d} = \min \left(e^{-1},\int|x-y|^{2}\,d\Pi_{t}\right)^{\frac{1}{2}}$.
\end{proof}

\subsection{Existence of the curve of transport plans}\label{subsec:existence_plans}
In this section we prove \cref{lem:existence_plan} in the case $2\leq \alpha \leq d$. That is, the existence of a curve of transport plans $t\mapsto \Pi_t\in \Gamma(f_t,g_t)$ such that $\Pi_0 = \Pi_0^*$ and
\begin{equation}\label{eq:plan}
  \frac{d}{dt}\int\varphi\,d\Pi_{t} = \int
  \left[\Big(\overline{c}+(\overline{b}*\Pi_{t}^{*})\Big)\cdot\nabla_{x,y}\varphi 
  + \Big(\overline{M}+(\overline{A}*\Pi_{t}^{*})\Big):D_{x,y}^{2}\varphi\right]\,d\Pi_{t},
\end{equation}
where $\Pi^*_t\in \Gamma(f_t,g_t)$ is the optimal transport plan. Several elements of this proof already appeared in \cite{fournier2009well_posedness_soft_potentials,fournier2010uniqueness_Coulomb} but we include all the details for the sake of completeness. We use that \cref{eq:plan} has a probabilistic counterpart given by the pair of stochastic differential equations
\begin{align}
  X_{t}&\nonumber
  =X_{0}+\int_{0}^{t}c(X_{s})\,ds+\sqrt 2\int_{0}^{t}\mathfrak m(X_{s})\,dB_s\\
  &+\int_{0}^{t}\int_{\mathbb{R}^{n}}b(X_{s}-x_{*})f_{s}(x_{*})\,dx_{*}\,ds
  +\sqrt 2\int_{0}^{t}\int_{\mathbb{R}^{n}\times\mathbb{R}^{n}}\sigma(X_{s}-x_{*})\,W(dx_{*},dy_{*},ds),
  \label{eq:SDEX}\\
  Y_{t} &\nonumber
  =Y_{0}+\int_{0}^{t}c(Y_{s})\,ds
  +\sqrt 2\int_{0}^{t}\mathfrak m(Y_{s})\,dB_s\\
  &+\int_{0}^{t}\int_{\mathbb{R}^{n}}b(Y_{s}-y_{*})g_{s}(y_{*})\,dy_{*}\,ds
  +\sqrt 2\int_{0}^{t}\int_{\mathbb{R}^{n}\times\mathbb{R}^{n}}\sigma(Y_{s}-y_{*})\,W(dx_{*},dy_{*},ds),
  \label{eq:SDEY}
\end{align}
where $B$ is an $n$-dimensional Brownian motion, the initial data are chosen with
\begin{equation*}
  \textnormal{Law}(X_{0})=f_{0},\quad\textnormal{Law}(Y_{0})=g_{0},\quad\textnormal{Law}(X_{0},Y_{0})=\Pi_{0}^{*},
\end{equation*}
and $W(dx,dy,dt)= (W_1,W_2,W_3)(dx,dy,dt)$ is a 3D white noise on $\mathbb{R}^{n}\times\mathbb{R}^{n}\times[0,T]$ independent from $B$ with covariance measure $d\Pi_{t}^{*}(x,y)\,dt$. See \cite{1986_Walsh_intro_SDEs} for details and definitions on white noise. Most importantly, this means that for any test functions $\varphi(x,y,t)$, $\psi(x,y,t)$, the processes
\begin{equation*}
  (\xi_{t}^{i})_{i=1}^{d}:=\int_{0}^{t}\int_{\mathbb{R}^{n}\times\mathbb{R}^{n}}\varphi(x,y,s)\,W(dx,dy,ds),\quad
  (\zeta_{t}^{i})_{i=1}^{d}:=\int_{0}^{t}\int_{\mathbb{R}^{n}\times\mathbb{R}^{n}}\psi(x,y,s)\,W(dx,dy,ds)
\end{equation*}
are square integrable martingales with cross variation structure 
\begin{equation*}
  \left\langle \xi^{i},\zeta^{j}\right\rangle _{t}=\delta_{ij}\int_{0}^{t}\int_{\mathbb{R}^{n}\times\mathbb{R}^{n}}\varphi(x,y,s)\psi(x,y,s)\,d\Pi_{s}^{*}(x,y)\,ds,
\end{equation*}
which is all the information required to perform the subsequent computations. Now a simple application of Itô's formula for a test function $\varphi\in \mathcal C^2_b(\mathbb R^d\times \mathbb R^d)$ shows that if $X_t$, $Y_t$ are solutions of \cref{eq:SDEX,eq:SDEY} then $\Pi_t=\textnormal{Law}(X_t,Y_t)$ is a weak solution of \cref{eq:plan}. Therefore, since both equations are analogous, the problem reduces to showing that \cref{eq:SDEX} has a solution and that $\textnormal{Law}(X_t)=f_t$. We divide the proof into various steps. One of the main goals is showing existence of strong solutions through Picard iterations on the mapping
\begin{align*}
  \Phi&:L^{\infty}\left([0,T];L^{2}(\Omega)\right)\longrightarrow L^{\infty}\left([0,T];L^{2}(\Omega)\right),\\
  \Phi(X)_{t}&
  :=X_{0}+\int_{0}^{t}c(X_{s})\,ds
  +\int_{0}^{t}\int_{\mathbb{R}^{d}}b(X_{s}-x_{*})f_{s}(x_{*})\,dx_{*}\,ds\\
  &+\sqrt 2\int_{0}^{t}\mathfrak m(X_{s})\,dB_s+\sqrt 2\int_{0}^{t}\int_{\mathbb{R}^{d}\times\mathbb{R}^{d}}\sigma(X_{s}-x_{*})\,W(dx_{*},dy_{*},ds),
\end{align*}
where $\Omega$ denotes the probability space.

\textbf{Step 1. Mapping $\Phi$ is well defined.} We verify that $\sup_{[0,T]}\mathbb{E}\left[\left|\Phi(X)_{t}\right|^{2}\right]<\infty$. Using the global Lipschitz assumption \ref{ass:Lipschitz} on $c$, $\mathfrak m$ shows
\begin{align*}
  \mathbb{E}\left[\left|\Phi(X)_{t}\right|^{2}\right] &
  \leq2\,\mathbb{E}\left[|X_{0}|^{2}\right] + (t^{2}+t)C\left(1+\sup_{[0,t]}\mathbb{E}\left[|X_{s}|^{2}\right]\right) 	\\
  & +C\,\mathbb{E}\left[\left(\int_{0}^{t}\int\left|b(X_{s}-x_{*})\right|f_{s}(x_{*})\,dx_{*}\,ds\right)^{2}\right]
  +C\,\mathbb{E}\left[\int_{0}^{t}\int|\sigma(X_{s}-x_{*})|^{2}f_{s}(x_{*})\,dx_{*}\,ds\right].
\end{align*}
We bound the remaining two terms using assumption \ref{ass:growth} and \cref{eq:sing_int}, which show
\begin{equation*}
  \int\left|b(X_{s}-x_{*})\right|f_{s}(x_{*}),\int|\sigma(X_{s}-x_{*})|^{2}f_{s}(x_{*})\leq C\max_{1\leq i\leq k}(\Vert f_{s}^{(i)}\Vert_{L^p},1),
\end{equation*}
which implies
\begin{align*}
  \mathbb{E}\left[\left|\Phi(X)_{t}\right|^{2}\right]&
  \lesssim\mathbb{E}\left[|X_{0}|^{2}\right]+
  C(t^{2}+t)\left(1+\sup_{[0,t]}\mathbb{E}\left[|X_{s}|^{2}\right]\right)\\&+C\int_{0}^{T}\max_{1\leq i\leq k}(\Vert f_{s}^{(i)}\Vert_{L^p},1)\,ds
  +C\left(\int_{0}^{T}\max_{1\leq i\leq k}(\Vert f_{s}^{(i)}\Vert_{L^p},1)\,ds\right)^{2},
\end{align*}
showing that $\Phi$ is well defined.

\textbf{Step 2. Continuity estimate on $\Phi$.} Given $X,\tilde{X}\in L^{\infty}\left([0,T];L^{2}(\Omega)\right)$ we estimate
\begin{align*}
  \Phi(X)_{t}-\Phi(\tilde{X})_{t} &
  =\int_{0}^{t}\left(c(X_{s})-c(\tilde{X}_{s})\right)\,ds+
  \sqrt{2}\int_{0}^{t}\left(\mathfrak m(X_{s})-\mathfrak m(\tilde{X}_{s})\right)\,dB_s	\\
  & +\int_{0}^{t}\int\left(b(X_{s}-x_{*})-b(\tilde{X}_{s}-x_{*})\right)f_{s}(x_{*})\,dx_{*}\,ds\\
  & +\sqrt{2}\int_{0}^{t}\int\left(\sigma(X_{s}-x_{*})-\sigma(X_{s}-x_{*})\right)W(dx_{*},dy_{*},ds).
\end{align*}
Using Itô's formula gives
\begin{align*}
  \frac{1}{2}\frac{d}{dt}\mathbb{E}\left[|\Phi(X)_{t}-\Phi(\tilde{X})_{t}|^{2}\right] &
  =\mathbb{E}\left[\left(c(X_{t})-c(\tilde{X}_{t})\right)\left(\Phi(X)_{t}-\Phi(\tilde{X}_{t})\right)\right]
  +\mathbb{E}\left[|\mathfrak m(X_{t})-\mathfrak m(\tilde{X}_{t})|^{2}\right]	\\
  & +\mathbb{E}\left[\left(\Phi(X)_{t}-\Phi(\tilde{X})_{t}\right)\cdot\int\left(b(X_{t}-x_{*})-b(\tilde{X}_{t}-x_{*})\right)f_{t}(x_{*})\,dx_{*}\right]\\
  & +\mathbb{E}\left[\int|\sigma(X_{t}-x_{*})-\sigma(\tilde{X}_{t}-x_{*})|^{2}f(x_{*})\,dx_{*}\right].
\end{align*}
\cref{lem:A_b_estimates} bounds the last two terms:
\begin{align*}
  \int|b(X_{t}-x_{*})-b(\tilde{X}_{t}-x_{*})|f_{t}(x_{*})\,dx_{*}
  &\leq C\max_{1\leq i\leq k}(\Vert f_{s}^{(i)}\Vert_{L^p},1)\Psi\left(|X_{t}-\tilde{X}_{t}|\right),\\
  \int|\sigma(X_{t}-x_{*})-\sigma(\tilde{X}_{t}-x_{*})|^{2}f(x_{*})\, dx_*
  &\leq C\max_{1\leq i\leq k}(\Vert f_{s}^{(i)}\Vert_{L^p},1)\Psi\left(|X_{t}-\tilde{X}_{t}|^{2}\right).
\end{align*}
Also note that, since  $\Psi$  is increasing and satisfies $a\Psi(a)\leq\Psi(a^{2})$ we have for any $a,b\geq0$
\begin{equation*}
  a\Psi(b)\leq\max\left\{ a\Psi(a),b\Psi(b)\right\} \leq\Psi(a^{2})+\Psi(b^{2}).
\end{equation*}
Putting all of this together yields
\begin{align}\label{eq:diff_ineq_phi}
  & \frac{d}{dt}\mathbb{E}\left[|\Phi(X)_{t}-\Phi(\tilde{X})_{t}|^{2}\right]	\\
  \leq & \; C\max_{1\leq i\leq k}(\Vert f_{t}^{(i)}\Vert_{L^p},1)
  \left[\Psi\left(\mathbb{E}\left[|\Phi(X)_{t}-\Phi(\tilde{X}_{t})|^{2}\right]\right)
  +\Psi\left(\mathbb{E}\left[|X_{t}-\tilde{X}_{t}|^{2}\right]\right)\right].\nonumber
\end{align}
\textbf{Step 3. Continuity of $\Phi$.} The differential inequality \cref{eq:diff_ineq_phi} suggests that we should have continuity of $\Phi:L^{\infty}\left([0,T];L^{2}(\Omega)\right)\to L^{\infty}\left([0,T];L^{2}(\Omega)\right)$. This turns out to be true, although we do not prove it using \cref{eq:diff_ineq_phi}. Note hat the log-Lipschitz modulus of continuity $\Psi$  shows only due to the non-integrability of $|x^{i}-x_{*}^{i}|^{-d}$ in $\mathbb R^d$, which appears after applying Hölder's inequality with exponents $\frac{d}{d-\alpha}$, $\frac{d}{\alpha}$ on assumption \ref{ass:local_Lipchitz}. Therefore, if we consider a cut-off function
\begin{equation*}
  \Theta\in\mathcal{C}^{\infty}(\mathbb{R}),\quad0\leq\Theta\leq1,\quad\Theta|_{[-1,1]}\equiv0,\quad\Theta|_{[-2,2]^{c}}\equiv1,\quad\Theta_{\varepsilon}(x):=\Theta(\varepsilon^{-1}x),
\end{equation*}
and define a regularized map $\Phi_{\varepsilon}:L^{\infty}\left([0,T];L^{2}(\Omega)\right)\to L^{\infty}\left([0,T];L^{2}(\Omega)\right)$ where $b(x)$, $\sigma(x)$ are replaced by $\Theta_{\varepsilon}(x)b(x)$, $\Theta_{\varepsilon}(x)\sigma(x)$, then we would obtain an estimate of the form
\begin{equation*}
  \frac{d}{dt}\mathbb{E}\left[|\Phi_{\varepsilon}(X)_{t}-\Phi_{\varepsilon}(\tilde{X})_{t}|^{2}\right]\leq
   C_{\varepsilon}\max_{1\leq i\leq k}(\Vert f_{t}^{(i)}\Vert_{L^p},1)\left(\mathbb{E}\left[|\Phi_{\varepsilon}(X)_{t}-\Phi_{\varepsilon}(\tilde{X}_{t})|^{2}\right]+
   \mathbb{E}\left[|X_{t}-\tilde{X}_{t}|^{2}\right]\right),
\end{equation*}
where $C_{\varepsilon}\nearrow\infty$  as $\varepsilon\searrow 0$. Here Gronwall's inequality clearly implies the continuity of $\Phi_{\varepsilon}$ since
\begin{equation*}
  \mathbb{E}\left[|\Phi_{\varepsilon}(X)_{t}-\Phi_{\varepsilon}(\tilde{X})_{t}|^{2}\right]\leq 
  M_\varepsilon e^{M_\varepsilon}\sup_{[0,t]}\mathbb{E}\left[|X_{s}-\tilde{X}_{s}|^{2}\right],
  \quad M_\varepsilon:=C_{\varepsilon}\int_{0}^{T}\max_{1\leq i\leq k}(\Vert f_{s}^{(i)}\Vert_{L^{p}},1)\,ds.
\end{equation*}
Now we show uniform convergence $\Phi_{\varepsilon}\to\Phi$ using Itô's formula, assumption \ref{ass:growth} and \cref{eq:sing_int_eps}
\begin{align*}
  \mathbb{E}\left[\left|\Phi(X)_{t}-\Phi_{\varepsilon}(X)_{t}\right|^{2}\right] &
  \leq2\mathbb{E}\left[\left(\int_{0}^{t}\int_{|X_{s}-x_{*}|<2\varepsilon}\left|b(X_{s}-x_{*})\right|f_{s}(x_{*})\,dx_{*}\,ds\right)^{2}\right]	\\
  & +2\mathbb{E}\left[\left(\int_{0}^{t}\int_{|X_{s}-x_{*}|<2\varepsilon}\left|\sigma(X_{s}-x_{*})\right|^{2}f_{s}(x_{*})\,dx_{*}\,ds\right)\right]\\
  & \leq C\left(1+\left(\int_{0}^{T}\max_{1\leq i\leq k}(\Vert f_{s}^{(i)}\Vert_{L^p},1)\,ds\right)^{2}\right)\varepsilon^{2}.
\end{align*}
This proves the continuity of $\Phi:L^{\infty}\left([0,T];L^{2}(\Omega)\right)\to L^{\infty}\left([0,T];L^{2}(\Omega)\right)$. 

\textbf{Step 4. Existence of solutions through Picard iterations.} First, we restrict the time interval to ensure that iterations are uniformly bounded. In step 1 we proved an inequality of the form%
\begin{equation*}
  \mathbb{E}\left[\left|\Phi(X)_{t}\right|^{2}\right]\leq C_{1}+tC_{2}\sup_{[0,t]}\mathbb{E}\left[|X_{s}|^{2}\right],
\end{equation*}
where both $C_{1}$, $C_{2}$ depend on the constants from the assumptions and $C_{1}$ additionally depends on $\mathbb{E}\left[|X_{0}|^{2}\right]$, $\Vert f^{(i)}\Vert _{L^{1}_tL^p_{x^i}}$. Therefore, if we set
\begin{equation*}
  \mathcal{B}_{R}:=\left\{ X\in L^{\infty}\left([0,\tau];L^{2}(\Omega)\right):\sup_{[0,\tau]}\mathbb{E}\left[|X_{t}|^{2}\right]\leq R\right\} ,\;\tau=\frac{1}{2C_{2}},\quad R=2C_{1}
\end{equation*}
then $\Phi:\mathcal{B}_{R}\longrightarrow\mathcal{B}_{R}$, which ensures that the following Picard iterates are uniformly bounded:
\begin{equation*}
  X_{t}^{(0)}=X_{0},\quad X_{t}^{(m+1)}=\Phi(X^{(m)})_{t},\quad0\leq t\leq\tau.
\end{equation*}
If we now define $\rho_{m,\ell}(t)=\sup_{[0,t]}\mathbb{E}\left[\left|X_{s}^{(m+\ell)}-X_{s}^{(m)}\right|^{2}\right]$, step 2 implies
\begin{equation*}
  \rho_{m+1,\ell}(t)\leq C\int_{0}^{t}\max_{1\leq i\leq k}(\Vert f_{s}^{(i)}\Vert_{L^p},1)\left[\Psi\left(\rho_{m+1,\ell}(s)\right)+\Psi\left(\rho_{m,\ell}(s)\right)\right]\,ds.
\end{equation*}
Now, setting $\rho_{m}(t)=\sup_{\ell}\rho_{m,\ell}(t)$, $\rho(t)=\limsup_{m}\rho_{m}(t)$, we conclude that
\begin{equation*}
  \rho(t)\leq C\int_{0}^{t}\max_{1\leq i\leq k}(\Vert f_{s}^{(i)}\Vert_{L^p},1)\Psi\left(\rho(s)\right)\,ds,
\end{equation*}
which implies that $\rho(t)\equiv 0$ and thus
\begin{equation*}
  \lim_{m}\sup_{\ell}\sup_{[0,\tau]}\mathbb{E}\left[|X_{t}^{(m+\ell)}-X_{t}^{(m)}|^{2}\right]=0,
\end{equation*}
which precisely means that $\left(X^{(m)}\right)_{m}$ is Cauchy in $L^{\infty}\left([0,\tau];L^{2}(\Omega)\right)$, so there exists a process ${X\in L^{\infty}\left([0,\tau];L^{2}(\Omega)\right)}$ such that
\begin{equation*}
  \lim_{m}\sup_{[0,\tau]}\mathbb{E}\left[|X_{t}-X_{t}^{(m)}|^{2}\right]=0.
\end{equation*}
Moreover, since $\Phi$  is continuous with respect to $L^{\infty}\left([0,\tau];L^{2}(\Omega)\right)$, we may let $m\to\infty$ in $X^{(m+1)}=\Phi(X^{(m)})$ to conclude that $X$ is a strong solution of \cref{eq:SDEX} in the interval $[0,\tau]$. Finally, note that the choice of $\tau$  only depends on the constants from our assumptions, so we may choose $X_{\tau}$ as new initial data and iteratively extend the solution to the whole time interval $[0,T]$.

\textbf{Step 5. Uniqueness of solutions to \cref{eq:SDEX}.} Assume we have two solutions $X,\tilde{X}\in L^{\infty}\left([0,T];L^{2}(\Omega)\right)$ with $X_{t}=\Phi(X)_{t}$, $\tilde{X}_{t}=\Phi(\tilde{X})_{t}$. Then \cref{eq:diff_ineq_phi} becomes
\begin{equation*}
  \mathbb{E}\left[|X_{t}-\tilde{X}_{t}|^{2}\right]\leq\int_{0}^{t}\max_{1\leq i\leq k}(\Vert f_{s}^{(i)}\Vert_{L^p},1)\Psi\left(\mathbb{E}\left[|X_{s}-\tilde{X}_{s}|^{2}\right]\right)\,ds,
\end{equation*}
which implies that $\mathbb{E}\left[|X_{t}-\tilde{X}_{t}|^{2}\right]=0$ for all $t\geq0$.

\textbf{Step 6. Showing that $\textnormal{Law}(X_t)=f_t$.} A simple application of Itô's formula shows that $\mu_{t}:=\textnormal{Law}(X_{t})$ is a weak solution of
\begin{equation*}
  \partial_{t}\mu_{t}=\textnormal{div}_{x}\left(c\mu_{t}+(b*f_{t})\mu_{t}\right)+D_{x}^{2}:\left(M\mu_{t}+(A_{t}*f_{t})\mu_{t}\right),
\end{equation*}
which is essentially \cref{eq:gen_eq} but with the coefficients frozen to remove the nonlinearity. By definition, we have that $f_t$, $\mu_t$ solve the same linear equation with the same coefficients and the same initial condition. Then, we would immediately have $\mu_t = f_t$ if we are able to apply \cite[Theorem B.1]{HoroKara1990MartProbBoltz} or \cite[Theorems 4.1 and 5.1]{BhattKara1993InvarMeasures}, which provide a uniqueness criterion for weak solutions of
\begin{equation*}
  \partial_t \mu_t = \mathcal L_t^*\mu_t,
\end{equation*}
where $\mathcal L_t^*$ denotes the formal adjoint of a differential operator. We must verify the following:
\begin{enumerate}
  \item There is a common domain $\mathcal D$ for $\mathcal L_t$ which is a dense sub-algebra of $\mathcal C_0(\mathbb R^n)$ and such that $\mathcal L_t : \mathcal D\longrightarrow L^\infty(\mathbb R^n)$.
  \item For each $t$, $\mathcal L_t$ satisfies the maximum principle. 
  \item There is a countable subset $\mathcal{D}_{0}\subset\mathcal{D}$ such that, for any $f\in\mathcal{D}$ and $t\geq0$, there is $\{f_{m}\}_{m}\subset\mathcal{D}_{0}$ with $\left\{ (f_{m},\mathcal L_{t}f_{m})\right\} _{m}$ uniformly bounded and converging pointwisely to $(f,\mathcal L_{t}f)$.
  \item The martingale problem associated with $(\mathcal L_t,\delta_{x_0})$ is well posed for each $x_0\in \mathbb R^n$. That is, there is a process such that $X_0 = x_0$ and for any $\varphi\in \mathcal D$ we have that
  \begin{equation}\label{eq:mart_prob}
    \varphi\left(X_{t}\right)-\int_{0}^{t}(\mathcal L_{s}\varphi)(X_{s})\,ds
  \end{equation}
  is a margingale. Moreover, solutions are unique in law.
\end{enumerate}
Point (2) is immediate from the form of the operator
\begin{equation*}
  \mathcal{L}_{t}\varphi=\left(c+b*f_{t}\right)\cdot\nabla\varphi+\left(M+A_{t}*f_{t}\right):D^{2}\varphi.
\end{equation*}
We have already shown (4) as well, since solutions to \cref{eq:mart_prob} are represented by solutions of \cref{eq:SDEX}, for which we already have path-wise uniqueness, which implies uniqueness in law. As for (1), from assumption \ref{ass:growth} we have
\begin{equation}\label{eq:bounds_Lt}
  \left|(b*f_{t})(x)\right|,\left|(A_{t}*f_{t})(x)\right|\leq 
  C\max_{1\leq i\leq k}(\Vert f_{t}^{(i)}\Vert_{L^p},1);\qquad
  \left|c(x)\right|,\left|M(x)\right|\leq C\left(1+|x|^2\right),
\end{equation}
so we may pick $\mathcal D = \mathcal C^2_c(\mathbb R^n)$. Finally, for (3), for each $R>0$ we consider a countable subset $\mathcal D_R$ of $\left\{ \varphi\in\mathcal{C}_{c}^{2}(\mathbb{R}^{n}):\textnormal{supp}\,\varphi\subset B_{2R}\right\} $ that is dense in $\left\{ \varphi\in\mathcal{C}_{c}^{2}(\mathbb{R}^{n}):\textnormal{supp}\,\varphi\subset B_{R}\right\}$  and let $\mathcal{D}_{0}=\bigcup_{R\in\mathbb{Q}_{+}}\mathcal{D}_{R}$ Then, given $\varphi\in\mathcal{D}$, we may choose $R>0$ such that $\textnormal{supp}\,\varphi\subset B_{R}$, in which case there exists $(\varphi_{m})_{m}\subset\mathcal{D}_{R}\subset\mathcal{D}_{0}$ such that $\varphi_{m}\to\varphi$  in $\mathcal{C}^{2}(\mathbb{R}^{d})$. Since $\textnormal{supp}\,\varphi_{m}\subset B_{2R}$ and $\sup_{m}\left\Vert \varphi_{m}\right\Vert _{\mathcal{C}^{2}}<\infty$, from \cref{eq:bounds_Lt} we immediately obtain that $\sup_{m}\left\Vert \mathcal{A}_{t}\varphi_{m}\right\Vert _{L^{\infty}}<\infty$  and that $\mathcal{A}_{t}\varphi_{m}\overset{m}{\to}\mathcal{A}_{t}\varphi$  pointwisely. Therefore, we may apply the cited result to conclude $\mu_t = f_t$, finishing the proof of \cref{lem:existence_plan}.

\subsection{The case $0< \alpha < 2$}\label{subsec:soft_pot_modif}
This case must be treated separately due to the sign change of the exponents of assumption \ref{ass:growth}. Nevertheless, the previous argument can still be applied with suitable modifications. Note that the case $\alpha = 2$ is already covered by the earlier analysis, while for $\alpha = 0$ the coefficients $b$ and $\sigma$ are globally Lipschitz, yielding unconditional uniqueness and thus \cref{thm:gen_thm} follows immediately. Finally, when $\alpha = 1$, no modification is needed for the terms involving the drift $b$, since in this case $b$ is Lipschitz. Therefore, it suffices to adapt the estimates involving $\sigma$ in the range $0 < \alpha < 2$, and those involving $b$ in the range $0 < \alpha < 1$. We begin by adapting \cref{lem:A_b_estimates}.

\begin{lemma}\label{lem:sigma_est_mod_soft}
  If $0<\alpha <2$ then for each $0<\varepsilon<1$ we have
  \begin{align*}
    \int\left|\sigma(x-z)-\sigma(y-z)\right|^{2}f_{t}(z)\,dz
    &\leq C(1-\log\varepsilon)\max_{1\leq i\leq k}(\Vert f_{t}^{(i)}\Vert_{L^{p}},1)|x-y|^{2}\\
    &+C\varepsilon^{\frac{\alpha^{2}}{2}}\left(m_{2}(f)+|x|^{2}+|y|^{2}+\max_{1\leq i\leq k}\Vert f_{t}^{(i)}\Vert_{L^{p}}\right),
  \end{align*}
  where we remind that $m_2(f)$ denotes an upper bound of the second moments of $f$.
\end{lemma}
\begin{proof}
  We follow the same strategy dividing the integral into regions.

  \textit{Region 1:} $R_{1}=\left\{ \min_{i}|x^{i}-z^{i}|\leq\varepsilon\right\}$ can be reduced to $R_{1}=\left\{ |x^{1}-z^{1}|\leq\varepsilon\right\}$, where assumption \ref{ass:growth} yields
  \begin{equation*}
    \int_{R_{1}}\left|\sigma(x-z)-\sigma(y-z)\right|f_{t}\,dz\leq\sum_{i=1}^{k}\int\mathds{1}_{|x^{1}-z^{1}|\leq\varepsilon}\left(|x^{i}-z^{i}|^{2-\alpha}+|y^{i}-z^{i}|^{2-\alpha}\right)f_{t}\,dz.
  \end{equation*}
  Using Hölder's inequality twice with exponents $\left(\frac{2}{1-\alpha},\frac{2}{\alpha}\right)$ and $\left(p=\frac{d}{d-\alpha},\frac{d}{\alpha}\right)$ gives
  \begin{align*}
    \int\mathds{1}_{|x^{1}-z^{1}|\leq\varepsilon}|x^{i}-z^{i}|^{2-\alpha}f_{t}\,dz&
    \leq\left(\int\mathds{1}_{|x^{1}-z^{1}|\leq\varepsilon}f_{t}^{(1)}\,dz\right)^{\frac{\alpha}{2}}\left(\int|x^{i}-z^{i}|^{2}f^{(i)}_t\,dz^{(i)}\right)^{1-\frac{\alpha}{2}}\\
    &\leq C\varepsilon^{\frac{\alpha^{2}}{2}}\left(m_{2}(f)+|x|^{2}\right)^{1-\frac{\alpha}{2}}\Vert f_{t}^{(1)}\Vert_{L^{p}}^{\frac{\alpha}{2}}\\
    &\leq C\varepsilon^{\frac{\alpha^{2}}{2}}\left(m_{2}(f)+|x|^{2}+\Vert f_{t}^{(1)}\Vert_{L^{p}}\right).
  \end{align*}
  The same computation bounds the term with $|y^i-z^i|$, yielding
  \begin{equation*}
    \int_{R_{1}}\left|\sigma(x-z)-\sigma(y-z)\right|f_{t}\,dz\leq C\varepsilon^{\frac{\alpha^{2}}{2}}\left(m_{2}(f)+|x|^{2}+|y|^{2}+\max_{1\leq i\leq k}\Vert f_{t}^{(i)}\Vert_{L^{p}}\right).
  \end{equation*}
  \textit{Region 2:} $R_{2}=\left\{ \min_{i}|y^{i}-z^{i}|\leq\varepsilon\right\} $ is analogous to the previous case.

  \textit{Region 3:} $R_{3}=\left\{ \min_{i}|x^{i}-z^{i}|\geq\varepsilon,\min_{i}|y^{i}-z^{i}|\geq\varepsilon\right\}$ is bounded using assumption \ref{ass:local_Lipchitz} with the same computation used in Section \ref{subsec:coeff_estimates}, which gives
  \begin{equation*}
    \int_{R_{3}}\left|\sigma(x-z)-\sigma(y-z)\right|f_{t}\,dz\leq C(1-\log\varepsilon)\max_{1\leq i\leq k}(\Vert f_{t}^{(i)}\Vert_{L^{p}},1)|x-y|^{2},
  \end{equation*}
  finishing the proof.
\end{proof}
Now we deal with the drift terms.
\begin{lemma}\label{lem:b_est_mod_soft}
  If $0<\alpha<1$ then for each $0<\varepsilon<1$ we have
  \begin{align}
    \int\left|b(x-z)-b(y-z)\right|f_{t}(z)\,dz&\label{eq:b_lem_mod_soft1}
    \leq C(1-\log\varepsilon)\max_{i}(\Vert f_{t}^{(i)}\Vert_{L^{p}},1)|x-y|\\
    &+C\varepsilon^{\alpha^{2}}\left(m_{2}(f)+1+|x|+|y|\right),\nonumber
  \end{align}
  and we also have
  \begin{align}
    & \int\left|b(x-x_{*})-b(y-y_{*})\right||x-y|\,d\Pi_{t}\,d\Pi_{t}^{*}\label{eq:b_lem_mod_soft2}\\ 
    \leq & \; C(1-\log\varepsilon)\max_{1\leq i\leq k}(\Vert f_{t}^{(i)}\Vert_{L^{p}},\Vert g_{t}^{(i)}\Vert_{L^{p}},1)\int|x-y|^{2}\,d\Pi_{t}\nonumber	\\
    +&\; C\varepsilon^{\alpha^{2}}\left(m_{2}(f)+m_{2}(g)+\max_{1\leq i\leq k}(\Vert f_{t}^{(i)}\Vert_{L^{p}},\Vert g_{t}^{(i)}\Vert_{L^{p}},1)\right)\left(\int|x-y|^{2}\,d\Pi_{t}\right)^{\frac{1}{2}}.\nonumber
  \end{align}
\end{lemma}
The proof is totally analogous to \cref{lem:sigma_est_mod_soft}, so we omit it. Next, we explain the modifications required in Sections \ref{subsec:main_comput} and \ref{subsec:existence_plans}.

\textit{Main computation (Section \ref{subsec:main_comput}).} The drift term is still estimated using
\begin{align}
  \int\left|\sigma(x-x_{*})-\sigma(y-y_{*})\right|^{2}\,d\Pi_{t}\,d\Pi_{t}^{*}&\nonumber
  \lesssim\int\left(\int\left|\sigma(x-x_{*})-\sigma(x-y_{*})\right|^{2}f_{t}(x)\,dx\right)\,d\Pi_{t}^{*}\\
  &+\int\left(\int\left|\sigma(x-y_{*})-\sigma(y-y_{*})\right|^{2}g_{t}(y_{*})\,dy_{*}\right)\,d\Pi_t.\label{eq:a_term2_adaptation}
\end{align}
For the first term  \cref{lem:sigma_est_mod_soft} gives
\begin{align*}
  \int\left(\int\left|\sigma(x-x_{*})-\sigma(x-y_{*})\right|^{2}f_{t}(x)\,dx\right)\,d\Pi_{t}^{*}&
  \leq C(1-\log\varepsilon)\max_{1\leq i\leq k}(\Vert f_{t}^{(i)}\Vert_{L^{p}},1)\int|x-y|^{2}\,d\Pi_{t}	\\
  &+C\varepsilon^{\frac{\alpha^{2}}{2}}\left(m_{2}(f)+\max_{1\leq i\leq k}\Vert f_{t}^{(i)}\Vert_{L^{p}}\right),
\end{align*}
so we may pick ${\displaystyle \varepsilon^{\frac{\alpha^{2}}{2}}=\min\left(\int|x-y|^{2}\,d\Pi_t,e^{-1}\right)}$. The term \cref{eq:a_term2_adaptation} is estimated similarly. Finally, the drift term is bounded using \cref{eq:b_lem_mod_soft2} with ${\displaystyle \varepsilon^{\alpha^{2}}=\min\left(\int|x-y|^{2}\,d\Pi_t,e^{-1}\right)}$.

\textit{Existence of the curve of transport plans (Section \ref{subsec:existence_plans}).}

In Step 1 we show that $\Phi:L^{\infty}\left([0,T];L^{2}(\Omega)\right)\longrightarrow L^{\infty}\left([0,T];L^{2}(\Omega)\right)$. In this case the bound is even simpler since $|\sigma|^2$ has subquadratic growth when $0<\alpha <2$ and $b$ has sub-linear growth when $0<\alpha <1$. For example, the drift term is estimated by
\begin{equation*}
  \mathbb{E}\left[\int\left|\sigma(X_{t}-x_{*})\right|^{2}f_{t}(x_{*})\,dx_{*}\right]
  \lesssim1+m_{2}(f)+\mathbb{E}\left[|X_{t}|^{2}\right].
\end{equation*}
This yields bounds on $\mathbb{E}\left[|\Phi(X)_{t}|^{2}\right]$ which allow a similar choice of small time interval in step 4 to guarantee that $\Phi:\mathcal{B}_{R}\longrightarrow\mathcal{B}_{R}$.

In Step 2 we obtain a differential inequality for
\begin{align*}
  \frac{1}{2}\frac{d}{dt}\mathbb{E}\left[|\Phi(X)_{t}-\Phi(\tilde{X})_{t}|^{2}\right] &
  =\mathbb{E}\left[\left(c(X_{t})-c(\tilde{X}_{t})\right)\left(\Phi(X)_{t}-\Phi(\tilde{X}_{t})\right)\right]
  +\mathbb{E}\left[|\mathfrak m(X_{t})-\mathfrak m(\tilde{X}_{t})|^{2}\right]	\\
  & +\mathbb{E}\left[\left(\Phi(X)_{t}-\Phi(\tilde{X})_{t}\right)\cdot\int\left(b(X_{t}-x_{*})-b(\tilde{X}_{t}-x_{*})\right)f_{t}(x_{*})\,dx_{*}\right]\\
  & +\mathbb{E}\left[\int|\sigma(X_{t}-x_{*})-\sigma(\tilde{X}_{t}-x_{*})|^{2}f(x_{*})\,dx_{*}\right].
\end{align*}
It is enough to restrict our attention to the case $\sup_{[0,\tau]}\mathbb E[|X_t|^2]\leq R$, where \cref{lem:sigma_est_mod_soft} gives
\begin{align*}
  \mathbb{E}\left[\int\left|\sigma(X_{t}-x_{*})-\sigma(\tilde{X}_{t}-x_{*})\right|^{2}f(x_{*})\,dx_{*}\right]&
  \leq C(1-\log\varepsilon)\max_{1\leq i\leq k}(\Vert f_{t}^{(i)}\Vert_{L^{p}},1)\mathbb{E}\left[|X_{t}-\tilde{X}_{t}|^{2}\right]\\
  &+C\varepsilon^{\frac{\alpha^{2}}{2}}\left(m_{2}(f)+\max_{1\leq i\leq k}(\Vert f_{t}^{(i)}\Vert_{L^{p}},1)\right).
\end{align*}
Now simply choose $\varepsilon^{\frac{\alpha^{2}}{2}}=\min\left(\mathbb{E}\left[|X_{t}-\tilde{X}_{t}|^{2}\right],e^{-\frac{\alpha^{2}}{2}}\right)$. For the drift term, \cref{lem:b_est_mod_soft} gives
\begin{align*}
  &\mathbb{E}\left[|\Phi(X)_{t}-\Phi(\tilde{X})_{t}|\int|b(X_{t}-x_{*})-b(X_{t}-x_{*})|f_{t}(x_{*})\,dx_{*}\right]	\\
  \leq & \,C\,(1-\log\varepsilon)\max_{1\leq i\leq k}(\Vert f_{t}^{(i)}\Vert_{L^{p}},1)\mathbb{E}\left[|X_{t}-\tilde{X}_{t}|^{2}\right]^{\frac{1}{2}}\mathbb{E}\left[|\Phi(X)_{t}-\Phi(\tilde{X})_{t}|^{2}\right]^{\frac{1}{2}}\\
  +& \,C\,\varepsilon^{\alpha^{2}}\left(m_{2}(f)+1\right)\mathbb{E}\left[|\Phi(X)_{t}-\Phi(\tilde{X}_{t})|^{2}\right]^{\frac{1}{2}}.
\end{align*}
Now set $\varepsilon^{\alpha^{2}}=\min\left(\mathbb{E}\left[|\Phi(X)_{t}-\Phi(\tilde{X})_{t}|^{2}\right]^{\frac{1}{2}},e^{-1}\right)$ and use $a\Psi(b)\leq\Psi(a^{2})+\Psi(b^{2})$ to deduce
\begin{align*}
  &\mathbb{E}\left[|\Phi(X)_{t}-\Phi(\tilde{X})_{t}|\int|b(X_{t}-x_{*})-b(X_{t}-x_{*})|f_{t}(x_{*})\,dx_{*}\right]\\
  \leq&\;C\left(1+m_{2}(f)+\max_{1\leq i\leq k}(\Vert f_{t}^{(i)}\Vert_{L^{p}},1)\right)
  \left(\Psi\left(\mathbb{E}\left[|X_{t}-\tilde{X}_{t}|^{2}\right]\right)+\Psi\left(\mathbb{E}\left[|\Phi(X)_{t}-\Phi(\tilde{X})_{t}|^{2}\right]\right)\right).
\end{align*}
The rest of the argument from Section \ref{subsec:existence_plans} plays out without relevant modifications.

\section{Symmetrization method}\label{sec:symmetrization}
This method was explained in detail in \cite{paper1}, so we spare some of the details. Especially since the computations for the fuzzy Landau and multiespecies equations have very similar structures. 

\subsection{Fuzzy Landau equation}\label{subsec:sym_fuzzy}
In this section we derive the following result, which can be proven rigorously using techniques from Section \ref{sec:gen_thm}.
\begin{theorem}\label{thm:fuzzy_sym_method}
  Let $f,g$ be weak solutions of the fuzzy Landau equation for $-3<\gamma \leq -2$ with bounded second moments. Let $p>\frac{3}{3+\gamma}$ and assume that $f,g\in L^{\alpha_{p,\gamma}}_tL^p_vL^1_x$. That is, we have
  \begin{equation*}
    \int_{0}^{T}\left(\Vert\tilde{f}_{t}\Vert_{L_{v}^{p}}^{\alpha_{p,\gamma}}+\Vert\tilde{g}_{t}\Vert_{L_{v}^{p}}^{\alpha_{p,\gamma}}\right)\,dt<\infty,
    \qquad\tilde{f}(v):=\int f(x,v)\,dx,
    \quad\alpha_{p,\gamma}:=-\frac{\gamma p}{3(p-1)}\in(0,1).
  \end{equation*}
  Then we have
  \begin{equation*}
    d_{2}^{2}(f_{t},g_{t})\leq d_{2}^{2}(f_{0},g_{0})\exp\left\{ C\int_{0}^{t}\max(\Vert\tilde{f}_{t}\Vert_{L_{v}^{p}}^{\alpha_{p,\gamma}},\Vert\tilde{g}_{t}\Vert_{L_{v}^{p}}^{\alpha_{p,\gamma}},1)\,dt\right\},
  \end{equation*}
  where $C>0$ depends on $\gamma$ and $\Vert \kappa\Vert_{\textnormal{Lip}}$.
\end{theorem}
We use the formulation of the $2$-Wasserstein distance in terms of a Hamilton-Jacobi equation:
\begin{equation}\label{eq:HJ_formulation2}
  d_{2}^{2}(f,g)=\sup\left\{ \int_{\mathbb R^d} u_1(v)g(v)\,dv-\int_{\mathbb R^d} u_0(v)f(v)\,dv:\partial_{s}u_s+\frac{1}{2}\left|\nabla u_s\right|^{2}=0\right\}.
\end{equation}
It is well known that an optimal $u$ generates the $2$-Wasserstein geodesic $(\rho_s)_{0\leq s\leq1}$ connecting $f$ and $g$ through the continuity equation
\begin{equation}\label{eq:CE}
  \partial_s \rho_s + \textnormal{div}(\rho_s \nabla u_s) = 0,\qquad \rho_0 = f,\; \rho_1 = g.
\end{equation}
See \cite{villani2003topicsOT} for more details on this formulation. Let us denote the collision operator of the fuzzy Landau equation as $q$:
\begin{equation*}
  q(f)(x,v) = \textnormal{div}_v\int_{\mathbb R^6} \kappa(x-x_*)\Phi(v-v_*) (\nabla_v-\nabla_{v_*})\big(f(x,v) f(x_*,v_*)\big)\,dx_*\,dv_*.
\end{equation*}
Simply differentiating in \cref{eq:HJ_formulation2} gives
\begin{align}
  \frac{d}{dt}d_{2}^{2}(f,g) & =\int u_1\partial_{t}g\,dx\,dv-\int u_0\partial_{t}f\,dx\,dv	\nonumber\\
  & =-\int u_1\left(v\cdot\nabla_{x}g\right)\,dx\,dv+\int u_0\left(v\cdot\nabla_{x}f\right)\,dx\,dv \label{eq:transport_term}\\
  & +\int u_1q(g)\,dx\,dv-\int u_0q(f)\,dx\,dv\label{eq:collision_term}.
\end{align}
The contribution \eqref{eq:transport_term} coming from the transport term can be controlled directly in terms of the Wasserstein distance. To see this, recall that the Hopf–Lax formula for the Hamilton-Jacobi equation (see Section 5.4.6 of \cite{villani2003topicsOT}) gives
\begin{equation*}
  u_0(x,v)=\phi(x,v)-\frac{|x|^{2}+|v|^{2}}{2},\qquad u_1(x,v)=\frac{|x|^{2}+|v|^{2}}{2}-\psi(x,v),
\end{equation*}
where $\psi = \phi^*$ is the Legendre transform of $\phi$, which provides the inverse optimal transport map $(\nabla \phi)^{-1} = \nabla \psi$. Then, using the change of variables $(y,w)=\nabla_{x,v}\phi$ gives
\begin{align*}
  (\ref{eq:transport_term}) & =\int w\cdot\left(y-\nabla_{y}\phi\right)g\,dy\,dw+\int v\cdot\left(x-\nabla_{x}\phi\right)f\,dx\,dv \\
  & =\int\nabla_{v}\phi\cdot\left(\nabla_{x}\phi-x\right)f\,dx\,dv+\int v\cdot\left(x-\nabla_{x}\phi\right)f\,dx\,dv\\
  & =\int\left(\nabla_{x}\phi-x\right)\left(\nabla_{v}\phi-v\right)\,f\,dx\,dv\leq\frac{1}{2}d_{2}^{2}(f,g).
\end{align*}
All in all, we have
\begin{equation*}
  \frac{d}{dt}d_{2}^{2}(f,g) \leq \frac{1}{2}d_2^2(f,g)
  +\int u_1q(g)\,dx\,dv-\int u_0q(f)\,dx\,dv.
\end{equation*}
Therefore, we may assume that $f$ and $g$ satisfy a transport free fuzzy Landau equation
\begin{equation}\label{eq:transport_free}
  \partial_t f = q(f),
  \qquad \partial_tg = q(g)
\end{equation}
and estimate $\frac{d}{dt}d_2^2(f,g)$. The first step is to establish a lifting lemma, which reduces the problem to estimating the evolution of the Wasserstein distance for a linear equation associated with the following \emph{lifted operator}:
\begin{equation*}
  Q(F):=\kappa(x-x_*)\left(\textnormal{div}_{v}-\textnormal{div}_{v_*}\right)\Big[\Phi(v-v_*)\left(\nabla_{v}-\nabla_{v_*}\right)F\Big].
\end{equation*}
The operator $Q$ acts on functions of $\mathbb R^6\times \mathbb R^6$ and is degenerate elliptic. It is connected with the Landau collision operator through tensorizaton and the projection in the $x$ and $v$ variables:
\begin{equation}\label{eq:Q_q}
  q(f)=\int_{\mathbb{R}^{6}}Q(f\otimes f)\,dx_*\,dv_*,\qquad(f\otimes f)(x,v,x_*,v_*)=f(x,v)f(x_*,v_*).
\end{equation}
Most importantly, we have the following lifting property:
\begin{lemma}\label{lem:lifting_fuzzy}
  Let $f,g$ be two solutions of (\ref{eq:transport_free}). Then we have
  \begin{equation*}
    \frac{d}{dt}d_2^2(f,g) = \frac{1}{2} \left.\frac{d}{d\tau}\right|_{\tau = 0} d_2^2(\tilde F,\tilde G)
  \end{equation*}
  where $\tilde F$, $\tilde G$ are solutions in $\mathbb R^6 \times \mathbb R^6$ of 
  \begin{equation*}
    \begin{cases}
      \partial_{\tau}\tilde{F}=Q(\tilde{F})\\
      \tilde{F}|_{\tau=0}=F,
    \end{cases},\begin{cases}
      \partial_{\tau}\tilde{G}=Q(\tilde{G})\\
      \tilde{G}|_{\tau=0}=G,
    \end{cases},
  \end{equation*}
  with $F = f_t\otimes f_t$, $G = g_t\otimes g_t$.
\end{lemma}
\begin{proof}
  We use formulation \cref{eq:HJ_formulation2} and consider the geodesic between $f$ and $g$ given by \cref{eq:CE}. It is know that the formulation \cref{eq:HJ_formulation2} applied to $d_2^2(F,G)$, gives an optimal $U$ which decomposes as
  \begin{equation*}
    U_s(x,v,x_*,v_*)=u_s(x,v)+u_s(x_*,v_*),
  \end{equation*}
  where $u$ is optimal for $d_2^2(f,g)$. Then the lifting property follows from \cref{eq:Q_q}
  \begin{align*}
    \frac{d}{dt}d_2^2(f,g)  
    & =\int_{\mathbb{R}^{6}}u(x,v,1)q(g)\,dx\,dv-\int_{\mathbb{R}^{6}}u(x,v,0)q(f)\,dx\,dv\\
    & =\iint_{\mathbb{R}^{6}\times \mathbb{R}^{6}}u(x,v,1)Q(G)\,dx\,dv\,dx_*\,dv_*
    -\iint_{\mathbb{R}^{6}\times \mathbb{R}^{6}}u(x,v,0)Q(F)\,dx\,dv\,dx_*\,dv_*\\
    &=\frac{1}{2}\iint_{\mathbb{R}^{6}\times \mathbb{R}^{6}}U(\cdot,1)Q(G)-\frac{1}{2}\iint_{\mathbb{R}^{6}\times \mathbb{R}^{6}}U(\cdot,0)Q(F)
    =\frac{1}{2} \left.\frac{d}{d\tau}\right|_{\tau = 0} d_2^2(\tilde F,\tilde G),
  \end{align*}
  where in the last line we have used that $Q(F),Q(G)$ are symmetric in $(x,v)$, $(x_*,v_*)$.
\end{proof}
Next, it is convenient to re-write the lifted operator in terms of the vector fields
\begin{equation}\label{eq:defn_bj}
  b_{1}(v-v_*):=\begin{pmatrix}0\\
  v^{3}_*-v^{3}\\
  v^{2}-v_*^{2}
  \end{pmatrix},\; b_{2}(v-v_*):=
  \begin{pmatrix}v^{3}-v_*^{3}\\
  0\\
  v_*^{1}-v^{1}
  \end{pmatrix},\; b_{3}(v-v_*):=\begin{pmatrix}v_*^{2}-v^{2}\\
  v^{1}-v_*^{1}\\
  0
  \end{pmatrix}.
\end{equation}
These satisfy the properties 
\begin{equation*}
  \sum_{j=1}^{3}b_{j}(z)\otimes b_{j}(z)=\left|z\right|^{2}\mathbb{P}(z),\qquad\textnormal{span}\left\{ b_{j}(z)\right\} =\left\langle z\right\rangle ^{\perp},\qquad\textnormal{div}\,b_{j}=0.
\end{equation*}
Using these properties it is easy to check tha the lifted operator can be re-written as
\begin{equation}\label{eq:expanded_Q_fuzzy}
  Q(F)= r^\gamma \sum_{j=1}^3 \left(\tilde b_j \cdot \nabla \left(\tilde b_j \cdot \nabla F\right)\right).
\end{equation}
$\nabla=\nabla_{x,v,x_*,v_*}$, and we have introduced 
\begin{equation*}
  r := |v-v_*|,\qquad
  \tilde{b}_{j}(v-v_*):=\begin{pmatrix}0\\
  b_{j}(v-v_*)\\
  0\\
  -b_{j}(v-v_*)
  \end{pmatrix}\in\mathbb{R}^{12},\qquad j=1,2,3.
\end{equation*}
Now using \ref{lem:lifting_fuzzy} and integrating by parts yields
\begin{equation*}
  2\frac{d}{dt} d_{2}^{2}(f,g) =\int_{0}^{1}\int_{\mathbb R^6\times \mathbb R^6}\left[\nabla U\cdot\nabla Q(U)-\frac{1}{2}Q\big(\left|\nabla U\right|^{2}\big)\right]P,
\end{equation*}
where $P=\rho\otimes \rho$ is the Wasserstein geodesic between $F = f\otimes f$ and $G = g\otimes g$. We use the expression \cref{eq:expanded_Q_fuzzy} to expand this into
\begin{align}
  2\frac{d}{dt}d_{2}^{2}(f,g) & 
  =\sum_{j}\int_{0}^{1}\int_{\mathbb R^6\times \mathbb R^6}\nabla U\cdot\nabla\left(\tilde{b}_{j}\cdot\nabla\left(\tilde{b}_{j}\cdot\nabla U\right)\right) r^\gamma \kappa P
  \label{eq:term1_fuzzy}\\
  & -\frac{1}{2}\sum_{j}\int_{0}^{1}\int_{\mathbb R^6\times \mathbb R^6}\left(\tilde{b}_{j}\cdot\nabla\left(\tilde{b}_{j}\cdot\nabla\left|\nabla U\right|^{2}\right)\right) r^\gamma \kappa P \label{eq:term2_fuzzy}\\
  & +\gamma\sum_{j}\int_{0}^{1}\int_{\mathbb R^6\times \mathbb R^6}\left(\hat{n}\cdot\nabla U\right)\left(\tilde{b}_{j}\cdot\nabla\left(\tilde{b}_{j}\cdot\nabla U\right)\right)r^{\gamma-2}\kappa P\nonumber \\
  & +\sum_{j}\int_{0}^{1}\int_{\mathbb R^6\times \mathbb R^6}\left(\nabla \kappa\cdot\nabla U\right)\left(\tilde{b}_{j}\cdot\nabla\left(\tilde{b}_{j}\cdot\nabla U\right)\right) r^\gamma P,\nonumber 
\end{align}
where we have introduced the vector $\hat n\in \mathbb R^6\times \mathbb R^6$ given by
\begin{equation*}
  \nabla r^\gamma=\gamma\, r^{\gamma-2}\;\hat{n},\qquad\hat{n}=\begin{pmatrix}0\\
  v-v_*\\
  0\\
  v_*-v
  \end{pmatrix}\perp\tilde{b}_{j}.
\end{equation*}
Next, we show that, as it happens for the heat equation, $(\ref{eq:term1_fuzzy})+(\ref{eq:term2_fuzzy})$ is negative. To do this we expand the dot products in terms of the following orthonormal basis of $\mathbb R^6\times \mathbb R^6$:
\begin{equation*}
  \tilde{e}_{i}=
  \frac{1}{\sqrt 2}\begin{pmatrix}0\\
  e_{i}\\
  0\\
  e_{i}
  \end{pmatrix},\qquad
  \hat{e}_{i}= \frac{1}{\sqrt 2}\begin{pmatrix}0\\
  e_{i}\\
  0\\
  -e_{i}
  \end{pmatrix},\qquad \xi_i=\begin{pmatrix}e_{i}\\
  0\\
  0\\
  0
  \end{pmatrix},\qquad \eta_i=\begin{pmatrix}0\\
  0\\
  e_{i}\\
  0
  \end{pmatrix},
\end{equation*}
where $e_{i}$ denotes the canonical basis of $\mathbb{R}^{3}$. In particular, we use that
\begin{align*}
  \left(\tilde{b}_{j}\cdot\nabla\left(\tilde{b}_{j}\cdot\nabla\left|\nabla U\right|^{2}\right)\right)&
  = \sum_{i=1}^{3}\left(\tilde{b}_{j}\cdot\nabla\left(\tilde{b}_{j}\cdot\nabla\left[(\tilde{e}_{i}\cdot\nabla U)^{2}\right]\right)\right)
  +\sum_{i=1}^{3}\left(\tilde{b}_{j}\cdot\nabla\left(\tilde{b}_{j}\cdot\nabla\left[(\hat{e}_{i}\cdot\nabla U)^{2}\right]\right)\right)	\\
  & + \sum_{i=1}^{3}\left(\tilde{b}_{j}\cdot\nabla\left(\tilde{b}_{j}\cdot\nabla\left[(\xi_i\cdot\nabla U)^{2}\right]\right)\right)
  +\sum_{i=1}^{3}\left(\tilde{b}_{j}\cdot\nabla\left(\tilde{b}_{j}\cdot\nabla\left[(\eta_i\cdot\nabla U)^{2}\right]\right)\right),
\end{align*}
and similarly expand
\begin{equation*}
  \nabla U\cdot\nabla\left(\tilde{b}_{j}\cdot\nabla\left(\tilde{b}_{j}\cdot\nabla U\right)\right) 
  = \sum_{i=1}^{3}(\tilde{e}_{i}\cdot\nabla U)\left(\tilde{e}_{i}\cdot\nabla\left(\tilde{b}_{j}\cdot\nabla\left(\tilde{b}_{j}\cdot\nabla U\right)\right)\right) + \cdots
\end{equation*}
%
This decomposes (\ref{eq:term1_fuzzy}) and (\ref{eq:term2_fuzzy}) as
\begin{align*}
  (\ref{eq:term1_fuzzy}) & = \sum_{i,j=1}^3
  \left[I_1(\tilde b_j,\tilde e_i) + I_1(\tilde b_j,\hat e_i) + I_1(\tilde b_j,\xi_i) + I_1(\tilde b_j,\eta_i)\right]	\\
  (\ref{eq:term2_fuzzy}) & = \sum_{i,j=1}^3
  \left[I_2(\tilde b_j,\tilde e_i) + I_2(\tilde b_j,\hat e_i) + I_2(\tilde b_j,\xi_i) + I_2(\tilde b_j,\eta_i)\right],
\end{align*}
where we have defined the quantities
\begin{equation*}
  I_{1}(e,b):=\int e\cdot\nabla\left(b\cdot\nabla\left(b\cdot\nabla U\right)\right)\left(e\cdot\nabla U\right)r^\gamma \kappa P,\quad 
  I_{2}(e,b):=-\frac{1}{2}\int b\cdot\nabla\left( b\cdot\nabla\left[(e\cdot\nabla U)^{2}\right]\right)r^\gamma \kappa P.
\end{equation*}
The advantage of using $\tilde{e}_{i}$, $\hat{e}_{i}$ over the canonical basis of $\mathbb{R}^{12}$ is the simplicity of the Lie brackets $[\tilde{e}_{i},\tilde{b}_{j}]$,$[\hat{e}_{i},\tilde{b}_{j}]$. We remind the reader that if $a$, $b$ are vector fields then $[a,b]$ verifies the identity
\begin{equation*}
  a\cdot\nabla\left(b\cdot\nabla u\right)-b\cdot\nabla\left(a\cdot\nabla u\right)=[a,b]\cdot\nabla u
\end{equation*}
for any smooth function $u$. The components of $[a,b]$ are defined as
\begin{equation*}
  [a,b]_{i}=a\cdot\nabla b_{i}-b\cdot\nabla a_{i}.
\end{equation*}
Let us know collect here all possible commutators. Firstly,
\begin{equation*}
  [\tilde{e}_{i},\tilde{b}_{j}]=[\xi_i,\tilde{b}_{j}]=[\eta_i,\tilde{b}_{j}]=0,\qquad i,j=1,2,3.
\end{equation*}
Secondly, we the values of $[\hat{e}_{i},\tilde{b}_{j}]$ are gathered in the following table:

\begin{table}[h]
  \centering

  \begin{tabular}{c|c|c|c}
    $[\hat{e}_{i},\tilde{b}_{j}]$ & $\tilde{b}_{1}$ & $\tilde{b}_{2}$ & $\tilde{b}_{3}$\tabularnewline
    \hline 
    $\tilde{e}_{1}$ & $0$ & $-2\hat{e}_{3}$ & $2\hat{e}_{2}$\tabularnewline
    \hline 
    $\tilde{e}_{2}$ & $2\hat{e}_{3}$ & $0$ & $-2\hat{e}_{1}$\tabularnewline
    \hline 
    $\tilde{e}_{3}$ & $-2\hat{e}_{2}$ & $2\hat{e}_{1}$ & $0$\tabularnewline
  \end{tabular}
\end{table}
In order to re-write $(\ref{eq:term1_fuzzy})+(\ref{eq:term2_fuzzy})$, we first work on $I_1(e,b)+I_2(e,b)$.
\begin{lemma}\label{lem:fuzzy_commutator_computation}
  Let $b$ be one of the vector fields $\tilde b_j$ and $e\in \mathbb R^6\times \mathbb R^6$ a fixed vector. Then
  \begin{align*}
    I_{1}(e,b)+I_{2}(e,b) & 
    =-\int\left(e\cdot\nabla\left(b\cdot\nabla U\right)\right)^{2} r^\gamma \kappa P
    -\int\left([e,b]\cdot\nabla U\right)\left(e\cdot\nabla U\right)\left(b\cdot\nabla P\right) r^\gamma \kappa\\
    & +\int\left([e,b]\cdot\nabla\left(b\cdot\nabla U\right)\right)\left(e\cdot\nabla U\right) r^\gamma \kappa P
    +\int\left(e\cdot\nabla\left(b\cdot\nabla U\right)\right)\left([e,b]\cdot\nabla U\right) r^\gamma \kappa P.
  \end{align*}
\end{lemma}
\begin{proof}
  First rewrite $I_1$ by commuting $e,b$, integrating by parts, remembering that $b\cdot\nabla$ only acts on the $v,v_*$ variables and that $\textnormal{div}_{v,v_*}(r^\gamma \kappa b)=0$
  \begin{align*}
    I_{1}(e,b) & =-\int\left(e\cdot\nabla\left(b\cdot\nabla U\right)\right)^{2} r^\gamma  \kappa P
    -\int\left(e\cdot\nabla\left(b\cdot\nabla U\right)\right)\left(b\cdot\nabla P\right)\left(e\cdot\nabla U\right) r^\gamma  \kappa\\
    & +\int\left([e,b]\cdot\nabla\left(b\cdot\nabla U\right)\right)\left(e\cdot\nabla U\right) r^\gamma  \kappa P
    +\int\left(e\cdot\nabla\left(b\cdot\nabla U\right)\right)\left([e,b]\cdot\nabla U\right) r^\gamma  \kappa P.
  \end{align*}
  Operate similarly on $I_2$ to obtain
  \begin{align*}
    I_{2}(e,b) 
    & =\int\left(b\cdot\nabla\left(e\cdot\nabla U\right)\right)\left(e\cdot\nabla U\right)\left(b\cdot\nabla P\right) r^\gamma  \kappa\\
    & =\int\left(e\cdot\nabla\left(b\cdot\nabla U\right)\right)\left(e\cdot\nabla U\right)\left(b\cdot\nabla P\right) r^\gamma  \kappa
    -\int\left([e,b]\cdot\nabla U\right)\left(e\cdot\nabla U\right)\left(b\cdot\nabla P\right) r^\gamma  \kappa,
  \end{align*}
  and note the cancellation with $I_{1}(e,b)$ yielding the desired expression.
\end{proof}
This allows us to re-write $\frac{d}{dt}d_2^2(F,G)$ in a more convenient way:
\begin{lemma}\label{lem:expression_dissip}
  Given solutions of the lifted equation $\partial_tF = Q(F)$, $\partial_tG = Q(G)$, we have
  \begin{align}
    \frac{d}{dt}d_2^2(F,G) &\nonumber
    = -\sum_{j}\int\left|\nabla\left(\tilde{b}_{j}\cdot\nabla U\right)\right|^{2}r^{\gamma}\kappa P	\nonumber\\
    &+\gamma\sum_{j}\int\left(\hat{n}\cdot\nabla U\right)\left(\tilde{b}_{j}\cdot\nabla\left(\tilde{b}_{j}\cdot\nabla U\right)\right)r^{\gamma-2}\kappa P \label{eq:fuzzy_bad_term1} \\
  & +\sum_{j}\int\left(\nabla \kappa\cdot\nabla U\right)\left(\tilde{b}_{j}\cdot\nabla\left(\tilde{b}_{j}\cdot\nabla U\right)\right) r^\gamma P, 
  \label{eq:fuzzy_bad_term2}
  \end{align}
\end{lemma}
\begin{proof}
  Using the decomposition of (\ref{eq:term1_fuzzy}), (\ref{eq:term2_fuzzy}) in terms of $I_1$, $I_2$ together with the fact that $[\hat e_i,\tilde b_j]$ are the only nonzero commutators yields
  \begin{align}
    (\ref{eq:term1_fuzzy})+(\ref{eq:term2_fuzzy}) & 
    =-\sum_{j}\int\left|\nabla\left(\tilde{b}_{j}\cdot\nabla U\right)\right|^{2}r^{\gamma}\kappa P \nonumber	\\
    & -\sum_{i,j}\int\left([\hat{e}_{i},\tilde{b}_{j}]\cdot\nabla U\right)\left(\hat{e}_{i}\cdot\nabla U\right)\left(\tilde{b}_{j}\cdot\nabla P\right)r^{\gamma}\kappa \label{eq:vanish1}\\
    & +\sum_{i,j}\int\left([\hat{e}_{i},\tilde{b}_{j}]\cdot\nabla\left(\tilde{b}_{j}\cdot\nabla U\right)\right)\left(\hat{e}_{i}\cdot\nabla U\right)r^{\gamma}\kappa P\label{eq:vanish2}\\
    & +\sum_{i,j}\int\left(\hat{e}_{i}\cdot\nabla\left(\tilde{b}_{j}\cdot\nabla U\right)\right)\left([\hat{e}_{i},\tilde{b}_{j}]\cdot\nabla U\right)r^{\gamma}\kappa P,\label{eq:vanish3}
  \end{align}
  Now observe that (\ref{eq:vanish1}) is symmetric in $\hat{e}_{i}$, $[\hat{e}_{i},\tilde{b}_{j}]$. Since the table of commutators is antisymmetric, we deduce that actually $(\ref{eq:vanish1})=0$. The same reasoning shows that $\ensuremath{(\ref{eq:vanish2})+(\ref{eq:vanish3})=0}$, proving the lemma.
\end{proof}
Next, we bound (\ref{eq:fuzzy_bad_term1}), (\ref{eq:fuzzy_bad_term2}) in terms of the Wasserstein distance and the negative term. Let us begin with the first term by recalling that $|\tilde b_j|\leq 2r$, so 
\begin{equation*}
  (\ref{eq:fuzzy_bad_term1}) \leq
  \frac{1}{2}\sum_{j}\int\left|\nabla\left(\tilde{b}_{j}\cdot\nabla U\right)\right|^{2}r^{\gamma}\kappa P
    +6\gamma^{2}\int\left|\hat{n}\cdot\nabla U\right|^{2}r^{\gamma-2}\kappa P
\end{equation*}
Now we use that $U = u\oplus u$, $P = \rho \otimes \rho$, which gives
\begin{equation*}
  \int\left|\hat{n}\cdot\nabla U\right|^{2}r^{\gamma-2}\kappa P \lesssim
  \int\left|\nabla_{v}u(x,v)\right|^{2}\rho(x,v)\Bigg(\underbrace{\int|v-v_*|^{\gamma}\rho(x_*,v_*)\,dx_*\,dv_*}_{=:(I_{\gamma}\rho)(x,v)}\Bigg)\,dx\,dv.
\end{equation*}
Thus, keeping in mind that $d_2^2(f,g)=\int_{\mathbb R^6}|\nabla u|^2\rho$, we have
\begin{equation*}
  (\ref{eq:fuzzy_bad_term1})\leq\frac{1}{2}\sum_{j}\int\left|\nabla\left(\tilde{b}_{j}\cdot\nabla U\right)\right|^{2}r^{\gamma}\kappa P+
  C\left\Vert I_{\gamma}\rho\right\Vert _{L^{\infty}} d_2^2(f,g).
\end{equation*}
Let us now turn our attention to (\ref{eq:fuzzy_bad_term2}). First we use Young's inequality
\begin{equation*}
  (\ref{eq:fuzzy_bad_term2}) \leq\frac{1}{2}\sum_{j}\int\left|\nabla\left(\tilde{b}_{j}\cdot\nabla U\right)\right|^{2}r^{\gamma}P
    +6\int\left|\nabla\kappa\cdot\nabla U\right|^{2}r^{\gamma+2}P.
\end{equation*}
Now introduce $u$, $\rho$ again to bound the last term
\begin{equation*}
  6\int\left|\nabla\kappa\cdot\nabla U\right|^{2}r^{\gamma+2}P\leq
  24\int\left|\nabla_{x}u(x,v)\right|^{2}\rho(x,v)\Bigg(\underbrace{\int|v-v_*|^{\gamma+2}\rho(x_*,v_*)\,dx_*\,dv_*}_{=(I_{\gamma+2}\rho)(x,v)}\Bigg)\rho(x,v)\,dx\,dv
\end{equation*}
Thus, we have
\begin{equation*}
  (\ref{eq:fuzzy_bad_term2}) \leq 
  \frac{1}{2}\sum_{j}\int\left|\nabla\left(\tilde{b}_{j}\cdot\nabla U\right)\right|^{2}r^{\gamma}\kappa P
  +C\left\Vert I_{\gamma+2}\rho\right\Vert _{L^{\infty}}d_2^2(f,g).
\end{equation*}
Putting it all together, we have
\begin{equation*}
  \frac{d}{dt}d_{2}^{2}(f,g)\le C\left(\left\Vert I_{\gamma}\rho\right\Vert _{L^{\infty}}+\left\Vert I_{\gamma+2}\rho\right\Vert _{L^{\infty}}\right)d_{2}^{2}(f,g).
\end{equation*}
Now we bound the singular integrals using Hölder's inequality with $p>\frac{3}{3+\gamma}$
\begin{equation*}
  \left\Vert I_{\gamma}\rho\right\Vert _{L^{\infty}}\leq 
    C_{p,\gamma}\left\Vert \rho\right\Vert _{L^{p}_vL^1_x}^{\alpha_{p,\gamma}},\qquad
    \left\Vert I_{\gamma+2}\rho\right\Vert_{L^{\infty}}\le C_{p,\gamma}\left\Vert \rho\right\Vert _{L^{p}_vL^1_x}^{\beta_{p,\gamma}},
\end{equation*}
where $0<\beta_{p,\gamma}=-\frac{(2+\gamma)p}{6(p-1)} < \alpha_{p,\gamma}=-\frac{\gamma p}{3(p-1)} <1$. Then \cref{thm:fuzzy_sym_method} follows immediately.

\subsection{Multiespecies Landau equation}\label{subsec:sym_multi}
In this section we derive the following result:
\begin{theorem}\label{thm:multi_sym_method}
  Let $(f^{(i)})_{i=1}^N$,$(g^{(i)})_{i=1}^N$ be weak solutions of the homogeneous multiespecies Landau equation for $-3<\gamma \leq 0$ with bounded second moments. Let $p>\frac{3}{3+\gamma}$ and assume that ${f^{(i)},g^{(i)}\in L^{\alpha_{p,\gamma}}_tL^p_{v^i}}$ where $\alpha_{p,\gamma}=-\frac{\gamma p}{3(p-1)}\in (0,1)$. Then we have
  \begin{equation*}
    \sum_{i=1}^{N}d_{2}^{2}(f_{t}^{(i)},g_{t}^{(i)})\leq\sum_{i=1}^{N}d_{2}^{2}(f_{0}^{(i)},g_{0}^{(i)})\exp\left\{ C\int_{0}^{t}\max_{1\leq i\leq N}(\Vert f_{t}^{(i)}\Vert_{L_{v^{i}}^{p}}^{\alpha_{p,\gamma}},\Vert g_{t}^{(i)}\Vert_{L_{v^{i}}^{p}}^{\alpha_{p,\gamma}},1)\,dt\right\},
  \end{equation*}
  where $C$ depends on $\gamma$, $N$, $\max_{ij}|c_{ij}|$, $\max_{i}|m_{i}|^{-1}$ and the second moments of $(f^{(i)})_{i=1}^{N}$, $(g^{(i)})_{i=1}^{N}$.
\end{theorem}
Just like \cref{thm:fuzzy_sym_method}, this can be made fully rigorous using the techniques from Section \ref{sec:gen_thm}. The first step for the derivation is to introduce the lifted operator, which acts on functions of $(\mathbb R^3)^{2N}$:
\begin{equation*}
  Q_{ij}(F):=c_{ij}\left(\frac{\textnormal{div}_{v^{i}}}{m_{i}}-\frac{\textnormal{div}_{v^{j}_*}}{m_{j}}\right)\left[\Phi(v^{i}-v^{j}_*)\left(\frac{\nabla_{v^{i}}}{m_{i}}-\frac{\nabla_{v^{j}_*}}{m_{j}}\right)F\right]
  ,\quad Q(F)=\sum_{i,j = 1}^NQ_{ij}(F),
\end{equation*}
which satisfies
\begin{equation*}
  \sum_{j=1}^{N}q_{ij}(f^{(i)},f^{(j)})=\int_{\left(\mathbb{R}^{3}\right)^{2N-1}}Q(F)\,dv^{-i},
\end{equation*}
where $dv^{-i}$ denotes integration with respect to all variables except $v^i$ and
\begin{equation*}
  F(v^{1},\ldots,v^{N},v_*^{1},\ldots,v_*^{N})=\bigotimes_{i=1}^N(f^{(i)}\otimes f^{(i)}):=\prod_{i=1}^{N}f^{(i)}(v^{i})f^{(i)}(v_*^{i}).
\end{equation*}
Moreover, we have following important symmetry:
\begin{equation}\label{eq:sym_Q_multi}
  Q(F)(v_*^{1},\ldots,v_*^{N},v^{1},\ldots,v^{N})=Q(F)(v^{1},\ldots,v^{N},v_*^{1},\ldots,v_*^{N}).
\end{equation}
This allows us to prove a similar lifting property.
\begin{lemma}
  Let $\left(f^{(i)}\right)_{i=1}^{N}$, $\left(g^{(i)}\right)_{i=1}^{N}$ be two solutions of the multiespecies Landau equation. Then we have
  \begin{equation*}
    \frac{d}{dt}\sum_{i=1}^Nd_2^2(f^{(i)}_t,g^{(i)}_t)=\frac{1}{2}\left.\frac{d}{d\tau}\right|_{\tau=0}d_{2}^{2}(\tilde{F}_\tau,\tilde{G}_\tau),
  \end{equation*}
  where $\tilde F$, $\tilde G$ are solutions in $(\mathbb R^3)^{2N}$ of 
  \begin{equation*}
    \begin{cases}
      \partial_{\tau}\tilde{F}=Q(\tilde{F})\\
      \tilde{F}(\tau=0)=F_t
    \end{cases},\begin{cases}
      \partial_{\tau}\tilde{G}=Q(\tilde{G})\\
      \tilde{G}(\tau=0)=G_t
    \end{cases},
  \end{equation*}
  and $F_t=\bigotimes_{i=1}^N(f^{(i)}_t\otimes f^{(i)}_t)$, $G_t=\bigotimes_{i=1}^N(g^{(i)}_t\otimes g^{(i)}_t)$.
\end{lemma}
\begin{proof}
  We differentiate the lifted equation using formulation \cref{eq:HJ_formulation2}:
  \begin{equation*}
    \left.\frac{d}{d\tau}\right|_{\tau=0}d_{2}^{2}(\tilde{F},\tilde{G}) = 
    \int_{\left(\mathbb{R}^{3}\right)^{2N}}U(1)Q(G)-\int_{\left(\mathbb{R}^{3}\right)^{2N}}U(0)Q(F).
  \end{equation*}
  It is not hard to see that when $F$, $G$ are tensor products, the optimal $U$ decomposes as
  \begin{equation}\label{eq:U_multiesp_decomp}
    U_s(v_{1},\ldots,v_{N},w_{1},\ldots,w_{N}) = \bigoplus_{i=1}^N(u^{(i)}_s\oplus u^{(i)}_s):=
    \sum_{i=1}^{N}\left(u^{(i)}_s(v^{i})+u^{(i)}_s(w^{i})\right),
  \end{equation}
  where each $u^{(i)}$ solves the same Hamilton-Jacobi equation in $\mathbb R^3$ and is optimal for $d_2^2(f^{(i)},g^{(i)})$. Now we simply compute
  \begin{align*}
    \frac{d}{dt}\sum_{i=1}^Nd_2^2(f^{(i)},g^{(i)}) &
    =\sum_{i=1}^{N}\left[\int_{\mathbb{R}^{3}}u^{(i)}_1(v^i)\sum_{j=1}^{N}q_{ij}(g^{(i)},g^{(j)})\,dv^{i}
    -\int_{\mathbb{R}^{3}}u^{(i)}_0(v^i)\sum_{j=1}^{N}q_{ij}(f^{(i)},f^{(j)})\,dv^{i}\right]	\\
    & =\sum_{i=1}^{N}\left[\int_{\left(\mathbb{R}^{3}\right)^{2N}}u^{(i)}_1(v_{i})Q(G)-\int_{\left(\mathbb{R}^{3}\right)^{2N}}u^{(i)}_0(v^{i})Q(F)\right]\\
    & =\frac{1}{2}\left[\int_{\left(\mathbb{R}^{3}\right)^{2N}}U_1Q(G)-\int_{\left(\mathbb{R}^{3}\right)^{2N}}U_0Q(F)\right],
  \end{align*}
  where in the last step we have used the symmetry property \cref{eq:sym_Q_multi}, which finishes the proof.
\end{proof}
This reduces the problem to estimating the growth of the $2$-Wasserstein distance for a linear equation. Then, integrating by parts we may re-write
\begin{equation}\label{eq:expr_before_bk}
  2\frac{d}{dt}\sum_{i=1}^Nd_2^2(f^{(i)},g^{(i)}) =\int_{0}^{1}\int_{(\mathbb R^3)^{2N}}\left[\nabla U\cdot\nabla Q(U)-\frac{1}{2}Q\left(|\nabla U|^{2}\right)\right]P.
\end{equation}
From here on, to simplify notation, we write $\int_{0}^{1}\int_{(\mathbb R^3)^{2N}}$ as a single integral. Next, we express the operator $Q$ in terms of vector fields in $(\mathbb R^3)^{2N}$. For this purpose we introduce the notation $r_{ij}=|v^{i}-v_*^{j}|$ and the $\left(\mathbb{R}^{3}\right)^{2N}$ vector $b_{k}^{i,j}=b_{k}^{i,j}(v^{i}-v_*^{j})$ which contains all $0$ except a $\frac{1}{m_{i}}b_{k}(v^{i}-v_*^{j})$ in position $v^{i}$ and a $-\frac{1}{m_{j}}b_{k}(v^{i}-v_*^{j})$ in position $v_*^{j}$ (see \cref{eq:defn_bj} for the definition of $b_k$):
\begin{equation*}
  b_{k}^{i,j}:= \Big( 0 , \ldots,  0, \overbrace{\frac{1}{m_{i}}b_{k}(v^{i}-v_*^{j})}^i, 
  0,  \ldots, 0, \overbrace{\frac{1}{m_{i}}b_{k}(v^{i}-w_*^{j})}^{N+j},  0, \ldots, 0 \Big)
\end{equation*}
We will also use that $\nabla r_{ij}^{2}=2n^{i,j}$, where 
\begin{equation*}
  n^{i,j}:= \Big( 0 , \ldots,  0, \overbrace{v^i-v_*^j}^i, 
  0,  \ldots, 0, \overbrace{-v^i+v_*^j}^{N+j},  0, \ldots, 0 \Big).
\end{equation*}
These vectors verify $n^{i,j}\cdot b_{k}^{i,j}=0$, and $\textnormal{div}\left(r_{ij}^\gamma b_{k}^{i,j}\right)=0$, which allows us to re-write $Q$ as
\begin{equation*}
  Q(H)=\sum_{i,j=1}^{N}c_{ij}r_{ij}^{\gamma}\sum_{k=1}^{3}\left(b_{k}^{i,j}\cdot\nabla\left(b_{k}^{i,j}\cdot\nabla H\right)\right).
\end{equation*}
Plugging this into (\ref{eq:expr_before_bk}) yields
\begin{align}
  2\frac{d}{dt}\sum_{i=1}^Nd_2^2(f^{(i)},g^{(i)}) & 
  =\sum_{i,j,k}c_{ij}\int\nabla\left(b_{k}^{i,j}\cdot\nabla\left(b_{k}^{i,j}\cdot\nabla U\right)\right)\cdot\nabla Ur_{ij}^{\gamma}P
  \label{eq:analog_heat_multiespecies1}\\
  & -\frac{1}{2}\sum_{i,j,k}c_{ij}\int\left(b_{k}^{i,j}\cdot\nabla\left(b_{k}^{i,j}\cdot\nabla|\nabla U|^{2}\right)\right)r_{ij}^{\gamma}P	
  \label{eq:analog_heat_multiespecies2}\\
  & +\frac{\gamma}{2}\sum_{i,j,k}c_{ij}\int\left(n^{i,j}\cdot\nabla U\right)\left(b_{k}^{i,j}\cdot\nabla\left(b_{k}^{i,j}\cdot\nabla U\right)\right)\,r_{ij}^{\gamma-2}P.\nonumber
\end{align}
Note that $(\ref{eq:analog_heat_multiespecies1})+(\ref{eq:analog_heat_multiespecies2})$ is analogous to the expression creating the dissipation for the heat equation in the introduction. We show that it contains a term with a negative sign by introducing the following basis of $(\mathbb R^3)^{2N}$:
\begin{equation*}
  \left\{ e_{l,q}\right\} _{1\leq l\leq N,1\leq q\leq3},\qquad\left\{ \xi_{l,q}\right\} _{1\leq l\leq N,1\leq q\leq3}.
\end{equation*}
Here $\{e_{l,q}\}_{q=1}^3$ (resp. $\{\xi_{l,q}\}_{q=1}^3$) is the canonical basis of $\mathbb{R}^3$ corresponding to the variable $v^i$ (resp. $v_*^i$).
The subsequent computation requires the commutators $[e_{l,q},b_{k}^{i,j}]$, $[\xi_{l,q},b_{k}^{i,j}]$, which we collect in the following lemma. The proof is an elementary (and rather tedious) computation.
\begin{lemma}
  The following commutators all banish:
  \begin{equation*}
    [e_{l,q},b_{k}^{i,j}]=0\textnormal{ if }l\neq i,\qquad[\xi_{l,q},b_{k}^{i,j}]=0\textnormal{ if }l\neq j.
  \end{equation*}
  In addition, $[e_{i,q},b_{k}^{i,j}]=-[\xi_{j,q},b_{k}^{i,j}]$ so we only require $[e_{i,q},b_{k}^{i,j}]$, which are displayed here:
  \renewcommand{\arraystretch}{3.5}
  \setlength{\extrarowheight}{-3pt}
  \begin{table}[H]
    \centering%
    \begin{tabular}{c|c|c|c} 
      $[e_{i,q},b_{k}^{i,j}]$ & $b_{1}^{i,j}$ & $b_{2}^{i,j}$ & $b_{3}^{i,j}$\tabularnewline
      \hline 
      $e_{i,1}$ & $0$ & ${\displaystyle -\frac{e_{i,3}}{m_{i}}+\frac{\xi_{j,3}}{m_{j}}}$ & $\displaystyle\frac{e_{i,2}}{m_{i}}-\frac{\xi_{j,2}}{m_{j}}$\tabularnewline
      \hline 
      $e_{i,2}$ & ${\displaystyle \frac{e_{i,3}}{m_{i}}-\frac{\xi_{j,3}}{m_{j}}}$ & $0$ & ${\displaystyle -\frac{e_{i,1}}{m_{i}}+\frac{\xi_{j,1}}{m_{j}}}$\tabularnewline
      \hline 
      $e_{i,3}$ & ${\displaystyle -\frac{e_{i,2}}{m_{i}}+\frac{\xi_{j,2}}{m_{j}}}$ & ${\displaystyle \frac{e_{i,1}}{m_{i}}-\frac{\xi_{j,1}}{m_{j}}}$ & $0$\tabularnewline
    \end{tabular}
  \end{table}
\end{lemma}
Now we have the following crucial lemma:
\begin{lemma}\label{lem:commutators}
  Consider fixed $i,j,k$ then we have
  \begin{equation*}
    (\ref{eq:analog_heat_multiespecies1})+(\ref{eq:analog_heat_multiespecies2}) \leq 
    -\frac{c_{ij}}{2} \int\left|\nabla\left(b_{k}^{i,j}\cdot\nabla U\right)\right|^{2}r_{ij}^{\gamma}P + C\max_{1\leq i\leq N}\left\Vert I_{\gamma}\rho_{i}\right\Vert _{L^{\infty}}\sum_{i=1}^Nd_2^2(f^{(i)},g^{(i)}),
  \end{equation*}
  where $C$ only depends on $\min_i m_i$, $\max_{i,j}c_{i,j}$ and $I_\gamma$ denotes the following singular integral
  \begin{equation}\label{eq:singular_integral}
    (I_{\gamma}\rho)(v):=\int_{\mathbb R^3}\left|v-v_*\right|^{\gamma}\rho(v_*)\,dv_*.
  \end{equation}
\end{lemma}
\begin{proof}
  This is an adaptation of \cref{lem:fuzzy_commutator_computation} and the subsequent computations. This case is a little more involved since now the commutators don't all cancel out, so there are extra terms to be bounded. If we consider $i,j,k$ fixed and omit the coefficient $c_{ij}$ in front then we have
  \begin{equation*}
    (\ref{eq:analog_heat_multiespecies1})  
    = \sum_{l,q}\left[I_{1}(b_{k}^{i,j},e_{l,q})+I_{1}(b_{k}^{i,j},\xi_{l,q})\right], \quad 
    (\ref{eq:analog_heat_multiespecies2}) 
    = \sum_{l,q}\left[I_{2}(b_{k}^{i,j},e_{l,q})+I_{2}(b_{k}^{i,j},\xi_{l,q})\right],
  \end{equation*}
  where we defined
  \begin{equation*}
    I_{1}(b,e)=\int\left(e\cdot\nabla\left(b\cdot\nabla\left(b\cdot\nabla U\right)\right)\right)\left(e\cdot\nabla U\right)r_{ij}^{\gamma}P,
    \quad I_{2}(b,e)=-\frac{1}{2}\int\left(b\cdot\nabla\left(b\cdot\nabla\left(e\cdot\nabla U\right)^{2}\right)\right)r_{ij}^{\gamma}P.
  \end{equation*}
  Repeating the computations from Lemma \ref{lem:fuzzy_commutator_computation} we obtain again
  \begin{align*}
    I_{1}(b,e)+I_{2}(b,e) & =-\int\left(e\cdot\nabla\left(b\cdot\nabla U\right)\right)^{2}r_{ij}^{\gamma}P-\int\left([e,b]\cdot\nabla U\right)\left(e\cdot\nabla U\right)\left(b\cdot\nabla P\right)r_{ij}^{\gamma}	\\
    & +\int\left([e,b]\cdot\nabla\left(b\cdot\nabla U\right)\right)\left(e\cdot\nabla U\right)r_{ij}^{\gamma}P+\int\left(e\cdot\nabla\left(b\cdot\nabla U\right)\right)\left([e,b]\cdot\nabla U\right)r_{ij}^{\gamma}P.
  \end{align*}
  This yields the following expression:
  \begin{align}
    (\ref{eq:analog_heat_multiespecies1})  + &(\ref{eq:analog_heat_multiespecies2})  =   
    -\int\left|\nabla\left(b_{k}^{i,j}\cdot\nabla U\right)\right|^{2}r_{ij}^{\gamma}P	\nonumber\\
    & -\sum_{l,q}\int\left[\left([e_{l,q},b_{k}^{i,j}]\cdot\nabla U\right)\left(e_{l,q}\cdot\nabla U\right)\left(b_{k}^{i,j}\cdot\nabla P\right)+\left([\xi_{l,q},b_{k}^{i,j}]\cdot\nabla U\right)\left(\xi_{l,q}\cdot\nabla U\right)\left(b_{k}^{i,j}\cdot\nabla P\right)\right]r_{ij}^{\gamma} 
    \raisetag{-6mm}\label{eq:commutators1}\\
    & +\sum_{l,q}\int\left[\left([e_{l,q},b_{k}^{i,j}]\cdot\nabla\left(b_{k}^{i,j}\cdot\nabla U\right)\right)\left(e_{l,q}\cdot\nabla U\right)+\left([\xi_{l,q},b_{k}^{i,j}]\cdot\nabla\left(b_{k}^{i,j}\cdot\nabla U\right)\right)\left(\xi_{l,q}\cdot\nabla U\right)\right]r_{ij}^{\gamma}P
    \raisetag{-6mm}\label{eq:commutators2}\\
    & +\sum_{l,q}\int\left[\left(e_{l,q}\cdot\nabla\left(b_{k}^{i,j}\cdot\nabla U\right)\right)\left([e_{l,q},b_{k}^{i,j}]\cdot\nabla U\right)+\left(\xi_{l,q}\cdot\nabla\left(b_{k}^{i,j}\cdot\nabla U\right)\right)\left([\xi_{l,q},b_{k}^{i,j}]\cdot\nabla U\right)\right]r_{ij}^{\gamma}P
    \raisetag{-6mm}\label{eq:commutators3}.
  \end{align} 
  Let us now expand (\ref{eq:commutators1}) for $k=1$ as an example:
  \begin{equation*}
    (\ref{eq:commutators1})=\alpha_{ij}\int\left[-\left(e_{i,2}\cdot\nabla U\right)\left(\xi_{j,3}\cdot\nabla U\right)\left(b_{1}^{i,j}\cdot\nabla P\right)+\left(e_{i,3}\cdot\nabla U\right)\left(\xi_{j,2}\cdot\nabla U\right)\left(b_{1}^{i,j}\cdot\nabla P\right)\right]r_{i,j}^{\gamma},
  \end{equation*}
  %
  %
  where we have introduced the coefficients $\alpha_{ij} = \frac{1}{m_{i}}-\frac{1}{m_{j}}$. Integrating by parts and commuting then shows that for $k=1$
  \begin{align*}
    (\ref{eq:commutators1}) & =  \alpha_{ij}\int\left(e_{i,2}\cdot\nabla\left(b_{1}^{i,j}\cdot\nabla U\right)\right)\left(\xi_{j,3}\cdot\nabla U\right)r_{ij}^{\gamma}P
    +\alpha_{ij}\int\left(\xi_{j,3}\cdot\nabla\left(b_{1}^{i,j}\cdot\nabla U\right)\right)\left(e_{i,2}\cdot\nabla U\right)r_{ij}^{\gamma}P	\\
    & -\alpha_{ij}\int\left(e_{i,3}\cdot\nabla\left(b_{1}^{i,j}\cdot\nabla U\right)\right)\left(\xi_{j,2}\cdot\nabla U\right)r_{ij}^{\gamma}P
    -\alpha_{ij}\int\left(\xi_{j,2}\cdot\nabla\left(b_{1}^{i,j}\cdot\nabla U\right)\right)\left(e_{i,3}\cdot\nabla U\right)r_{ij}^{\gamma}P\\
    & -\alpha_{ij}^{2}\int\left(e_{i,2}\cdot\nabla U\right)\left(\xi_{j,2}\cdot\nabla U\right)r_{ij}^{\gamma}P
    -\alpha_{ij}^{2}\int\left(e_{i,3}\cdot\nabla U\right)\left(\xi_{j,3}\cdot\nabla U\right)r_{ij}^{\gamma}P\\
    & -\alpha_{ij}^{2}\int\left(e_{i,2}\cdot\nabla U\right)^{2}r_{ij}^{\gamma}P
    -\alpha_{ij}^{2}\int\left(e_{i,3}\cdot\nabla U\right)r_{ij}^{\gamma}P.
  \end{align*}
  All the previous terms are treated similarly using Young's inequality with a parameter $\varepsilon>0$. For example, for the first one we have
  \begin{equation*}
    \alpha_{ij}\int\left(e_{i,2}\cdot\nabla\left(b_{1}^{i,j}\cdot\nabla U\right)\right)\left(\xi_{j,3}\cdot\nabla U\right)r_{ij}^{\gamma}P\leq\varepsilon\int\left|\nabla\left(b_{1}^{i,j}\cdot\nabla U\right)\right|^{2}r_{ij}^{\gamma}P+\frac{\alpha_{ij}^{2}}{4\varepsilon}\int\left(\xi_{j,3}\cdot\nabla U\right)^{2}r_{ij}^{\gamma}P.
  \end{equation*}
  Now expand the integral remembering that $U=\bigoplus_{i}(u^{(i)}\oplus u^{(i)})$ and $P=\bigotimes_{i}(\rho^{(i)}\otimes\rho^{(i)})$
  \begin{align*}
    \frac{\alpha_{ij}^{2}}{4\varepsilon}\int\left(\xi_{j,3}\cdot\nabla U\right)^{2}r_{ij}^{\gamma}P &
    =\frac{\alpha_{ij}^{2}}{4\varepsilon}\int\left(\partial_{3}u^{(j)}(v^{j}_*)\right)^{2}|v^{i}-v_*^{j}|^{\gamma}\rho^{(i)}(v^{i})\rho^{(j)}(v_*^{j})\,dv^{i}\,dv_*^{j}	\\
    & \leq\frac{\alpha_{ij}^{2}}{4\varepsilon}\Vert I_{\gamma}\rho^{(i)}\Vert _{L^{\infty}}d_{2}^{2}(f^{(j)},g^{(j)}),
  \end{align*}
  which shows
  \begin{align*}
    \alpha_{ij}\int\left(e_{i,2}\cdot\nabla\left(b_{1}^{i,j}\cdot\nabla U\right)\right)\left(\xi_{j,3}\cdot\nabla U\right)r_{ij}^{\gamma}P &
    \leq \varepsilon\int\left|\nabla\left(b_{1}^{i,j}\cdot\nabla U\right)\right|^{2}r_{ij}^{\gamma}P	\\
    & + \frac{\alpha_{ij}^{2}}{4\varepsilon}\Vert I_{\gamma}\rho^{(i)}\Vert _{L^{\infty}}d_{2}^{2}(f^{(j)},g^{(j)}).
  \end{align*}
  The same estimate applies to the rest of the terms in the previous expansion of (\ref{eq:commutators1}) and for $k=2,3$. All in all, we obtain
  \begin{equation*}
    (\ref{eq:commutators1}) \leq 6 \varepsilon\int\left|\nabla\left(b_{k}^{i,j}\cdot\nabla U\right)\right|^{2}r_{ij}^{\gamma}P
    +C\max_{1\leq i\leq N}\Vert I_{\gamma}\rho^{(i)}\Vert _{L^{\infty}}\sum_{i=1}^Nd_2^2(f^{(i)},g^{(i)}).
  \end{equation*}
  The term $(\ref{eq:commutators2}) + (\ref{eq:commutators3})$ is bounded very similarly. 
  All together, we have
  \begin{equation*}
    (\ref{eq:analog_heat_multiespecies1})+(\ref{eq:analog_heat_multiespecies2}) 
    \leq
    -\left(1-10\varepsilon\right)\int\left|\nabla\left(b_{k}^{i,j}\cdot\nabla U\right)\right|^{2}r_{ij}^{\gamma}P
    +C_{\varepsilon}\max_{1\leq i\leq N}\Vert I_{\gamma}\rho^{(i)}\Vert _{L^{\infty}}\sum_{i=1}^Nd_2^2(f^{(i)},g^{(i)}).
  \end{equation*}
  Picking $\varepsilon=\frac{1}{20}$ finishes the proof of the lemma after putting $c_{ij}$ in front again.
\end{proof}
Applying Lemma \ref{lem:commutators} to the expression (\ref{eq:analog_heat_multiespecies1}) yields
\begin{align}
  \frac{d}{dt}\sum_{i=1}^Nd_2^2(f^{(i)},g^{(i)}) & \leq -\frac{1}{4}\sum_{i,j,k}c_{ij} \int\left|\nabla\left(b_{k}^{i,j}\cdot\nabla U\right)\right|^{2}r_{ij}^{\gamma}P
   \nonumber	\\
  & +\frac{\gamma}{4}\sum_{i,j,k}c_{ij}\int\left(n^{i,j}\cdot\nabla U\right)\left(b_{k}^{i,j}\cdot\nabla\left(b_{k}^{i,j}\cdot\nabla U\right)\right)\,r_{ij}^{\gamma-2}P \label{eq:extra_term_multiespecies}\\
  & + C \max_{1\leq i\leq N}\left\Vert I_{\gamma}\rho^{(i)}\right\Vert _{L^{\infty}(\mathbb{R}^{3}\times[0,1])} \sum_{i=1}^Nd_2^2(f^{(i)},g^{(i)}).\nonumber
\end{align}
Now we repeat the computations from Section \ref{subsec:sym_fuzzy} to bound \cref{eq:extra_term_multiespecies},
which implies
\begin{equation*}
  \frac{d}{dt}\sum_{i=1}^Nd_2^2(f^{(i)},g^{(i)}) \leq C\max_{1\leq i\leq N}\Vert I_{\gamma}\rho^{(i)}\Vert _{L^{\infty}}\sum_{i=1}^Nd_2^2(f^{(i)},g^{(i)}).
\end{equation*}
Now we simply bound the singular integral using Hölder's inequality with $p>\frac{3}{3+\gamma}$, which yields
\begin{equation*}
  \Vert I_{\gamma}\rho^{(i)}\Vert _{L^{\infty}} \leq 
  C_{p,\gamma}\Vert \rho^{(i)}\Vert _{L^{p}}^{\alpha_{p,\gamma}},\qquad
  \alpha_{p,\gamma}=-\frac{\gamma p}{3(p-1)}\in (0,1),
\end{equation*}
which formally shows \cref{thm:multi_sym_method}.

\section*{Acknowledgments}
The author would like to thank Matias Delgadino and Maria Gualdani for their constant support and insightful discussions.

\bibliographystyle{plain}  
\bibliography{references}

\end{document}